\UseRawInputEncoding
\documentclass[11pt,reqno]{amsart}
\usepackage{hyperref}
\usepackage{url}
\usepackage{amssymb}

\newtheorem{theorem}{Theorem}

\newtheorem{lemma}{Lemma}

\theoremstyle{definition}

\newtheorem{rem}{Remark}
\numberwithin{equation}{section}
\numberwithin{theorem}{section}
\numberwithin{lemma}{section}
 \numberwithin{rem}{section}
\usepackage{tikz,enumerate}
\allowdisplaybreaks[4]
\begin{document}

\title[Modular Forms and $k$-colored Generalized Frobenius Partitions]
 {Modular Forms and $k$-colored Generalized Frobenius Partitions}

\author[Heng Huat Chan, Liuquan Wang, and Yifan Yang]{Heng Huat Chan, Liuquan Wang *, and Yifan Yang}
\address{Main address: Department of Mathematics, National University of Singapore,
Block S17, 10 Lower Kent Ridge Road,
Singapore 119076, Singapore}
\address{Temporary address: Fakult\"at f\"ur Mathematik, Universit\"at Wien,
Oskar-Morgenstern-Platz~1, A-1090 Vienna, Austria}
\email{matchh@nus.edu.sg}

\address{School of Mathematics and Statistics, Wuhan University, Wuhan 430072, Hubei, People's Republic of China}
\email{mathlqwang@163.com;wangliuquan@u.nus.edu;wanglq@whu.edu.cn}

\address{Department of Mathematics, National Taiwan University, Taipei, Taiwan 10617}
\email{yangyifan@ntu.edu.tw}



\dedicatory{Dedicated to Professor George E. Andrews on the occasion of his 80th birthday}
\thanks{* Corresponding author.}
\thanks{We corrected some typos appeared in the published version \cite{CWY-TAMS}. See the last section for details.}

\maketitle

\begin{abstract}
Let $k$ and $n$ be positive integers.
Let $c\phi_{k}(n)$ denote the number of $k$-colored generalized Frobenius partitions of $n$ and
$\mathrm{C}\Phi_k(q)$ be the generating function
of $c\phi_{k}(n)$. In this article, we study $\mathrm{C}\Phi_k(q)$  using the theory of modular forms and
discover new surprising properties of $\mathrm{C}\Phi_k(q)$.
\end{abstract}

\section{Introduction}

A partition $\pi$ of an integer $n$ is a sequence of non-increasing positive integers which add up to $n$. We denote the number of partitions of $n$ by $p(n)$. It is known that a partition $\pi$ of $n$  can be visualized using a Ferrers diagram by representing the positive integer $m$ of the $s$-th part by  $m$ dots on the $s$-th row. An example showing the pictorial representation of the partition
$4+4+4+2$ of the integer 14 is given in Figure \ref{fig1}.
\medskip

\begin{figure}[h]
\begin{center}
 \begin{tikzpicture}[scale=0.6]
 \filldraw (0,0) circle (3pt);
 \filldraw (1,0) circle (3pt);
 \filldraw (2,0) circle (3pt);
 \filldraw (3,0) circle (3pt);
 \filldraw (0,-1) circle (3pt);
 \filldraw (1,-1) circle (3pt);
  \filldraw (2,-1) circle (3pt);
  \filldraw (3,-1) circle (3pt);
   \filldraw (0,-2) circle (3pt);
 \filldraw (1,-2) circle (3pt);
  \filldraw (2,-2) circle (3pt);
  \filldraw (3,-2) circle (3pt);
    \filldraw (0,-3) circle (3pt);
 \filldraw (1,-3) circle (3pt);
\end{tikzpicture}
\end{center}
\caption{}\label{fig1}
\end{figure}
\medskip

From the Ferrers diagram of a partition, we can construct a $2$ by $d$ matrix by carrying out the following steps:
\begin{enumerate}[Step 1.]
\item Remove all the dots lying on the diagonal of the diagram.
\item Fill the first row of the matrix with entries $r_{1,j}$, where $r_{1,j}$ is the number of dots on the
$j$-th row that are to the right of the diagonal.
\item  Fill the second row of the matrix with entries $r_{2,j}$, where $r_{2,j}$ is the number of dots on the
$j$-th column that are below the diagonal.
\end{enumerate}
For example,
after Step 1, we obtained Figure \ref{fig2} from Figure \ref{fig1}.
\begin{figure}[h]
\begin{center}
 \begin{tikzpicture}[scale=0.6]
\draw[gray,very thin] (0,0) circle (3pt);
 \filldraw (1,0) circle (3pt);
 \filldraw (2,0) circle (3pt);
 \filldraw (3,0) circle (3pt);
 \filldraw (0,-1) circle (3pt);
\draw[gray,very thin] (1,-1) circle (3pt);
  \filldraw (2,-1) circle (3pt);
    \filldraw (3,-1) circle (3pt);
   \filldraw (0,-2) circle (3pt);
 \filldraw (1,-2) circle (3pt);
\draw[gray,very thin] (2,-2) circle (3pt);
\filldraw (3,-2) circle (3pt);
    \filldraw (0,-3) circle (3pt);
 \filldraw (1,-3) circle (3pt);
\end{tikzpicture}
\end{center}
\caption{}\label{fig2}
\end{figure}
Carrying out Steps 2 and 3, we arrive at the matrix
 $$\begin{pmatrix}
3 & 2 & 1 \\
3 &2 & 0
\end{pmatrix}.$$

It is clear that we can always construct a $2$ by $d$ matrix from any partition $\pi$ with $d$ dots along the diagonal of its Ferrers diagram
 and the matrix obtained  from a partition $\pi$ using the above procedures is
called a Frobenius symbol for the partition $\pi$. A Frobenius symbol, by construction, has
strictly decreasing entries on each row.

One way to find new functions that are similar to the partition function $p(n)$ is to start with a modified version of the Frobenius symbol.
In his 1984 AMS Memoir,  G.E. Andrews \cite[Section 4]{Andrews} introduced a generalized Frobenius symbol with at most $k$ repetitions for each integer by relaxing the ``strictly decreasing'' property and allowing at most $k$-repetitions of each positive integer in each row.
Andrews then used the generalized Frobenius symbol to define the generalized Frobenius partition of $n$. For a generalized
Frobenius symbol with entries $r_{i,j}, i=1,2, 1\leq j\leq d$, the generalized Frobenius partition of $n$ is
given by
$$n=d+\sum_{j=1}^{d}(r_{1,j}+r_{2,j}).$$ Andrews used the symbol $\phi_k(n)$ to denote
the number of such partitions of $n$.
As an example, we observe that $\phi_2(3)=5$ and these are given by the following generalized Frobenius  symbols with at most $2$ repetitions on each row:
 $$\begin{pmatrix} 2  \\ 0 \end{pmatrix},\quad \begin{pmatrix} 0  \\ 2  \end{pmatrix}, \quad \begin{pmatrix} 1 \\ 1 \end{pmatrix},
 \quad \begin{pmatrix} 1 & 0   \\ 0 & 0  \end{pmatrix},\quad \begin{pmatrix} 0 & 0   \\ 1 & 0  \end{pmatrix}.$$
Note that with this definition, $$\phi_1(n) = p(n).$$

There are at most $k$-repetitions in each row of a generalized Frobenius symbol. In order to restore
the ``strictly decreasing'' property of a Frobenius symbol from a generalized Frobenius symbol, Andrews colored
the repeated parts using ``colors'' denoted by $1,2,\cdots, k$ and imposed an ordering on these parts as follows:
\begin{equation}\label{new-ordering}0_1\prec 0_2\prec \cdots \prec 0_k\prec 1_1\prec 1_2\prec \cdots \prec 1_k \prec 2_1\prec 2_2\prec \cdots \prec 2_k\prec \cdots.\end{equation}
Here, we use ``$\prec$'' to differentiate the inequality from the usual inequality ``$<$''.
Andrews referred to a symbol obtained using $k$-colors in this way as a {$k$-colored} generalized Frobenius symbol.

Given a $k$-colored generalized Frobenius symbol
with entries
$$r_{i,j}\in \{\ell_c| \text{$\ell$ and $c$ are non-negative integers with $1\leq c\leq k$}\}$$
and
$$r_{i,j+1}\prec r_{i,j}, i=1,2\quad\text{and}\quad 1\leq j\leq d-1.$$
Andrews associated a $k$-colored generalized Frobenius partition of $n$ to a $k$-colored generalized Frobeinus symbol $(r_{i,j})_{2\times d}$  by setting
$$n=d+\sum_{j=1}^{d}(r_{1,j}+r_{2,j}),$$ where only the non-negative integer $\ell$ is added if $r_{i,j}=\ell_c.$
He used the symbol  $c\phi_k(n)$ to denote the number of such partitions of $n$.
Observe that when $k=1$, the $1$-colored generalized Frobenius symbols coincide with the Frobenius symbols and $c\phi_1(n)=p(n)$.
To help the reader understand $k$-colored generalized Frobenius symbols, we list the following $2$-colored generalized
Frobenius symbols which give rise to $2$-colored generalized Frobenius partitions of 2:
 \begin{align}\label{2-color-2}&\begin{pmatrix} 1_1\\  0_1 \end{pmatrix},\quad \begin{pmatrix} 1_1  \\ 0_2  \end{pmatrix},\quad \begin{pmatrix} 1_2 \\ 0_1 \end{pmatrix},\quad \begin{pmatrix} 1_2 \\ 0_2 \end{pmatrix},\\
 &\begin{pmatrix} 0_1 \\  1_1 \end{pmatrix},\quad \begin{pmatrix} 0_2 \\  1_1  \end{pmatrix},\quad \begin{pmatrix} 0_1 \\  1_2 \end{pmatrix},\quad \begin{pmatrix} 0_2 \\  1_2  \end{pmatrix},\quad \begin{pmatrix} 0_2  &  0_1 \\  0_2 & 0_1 \end{pmatrix}.\notag\end{align}
Note that there are altogether nine $2$-colored generalized Frobenius partitions of $2$ and hence, $$c\phi_2(2)=9.$$

The best way to study a new function such as the $k$-colored generalized Frobenius partition function $c\phi_k(n)$ is to study its generating function
$$\mathrm{C}\Phi_k(q):=\sum_{n=0}^\infty c\phi_k(n)q^n.$$
In \cite[Theorem 5.2]{Andrews},
Andrews showed that
\begin{equation}\label{notation}
\mathrm{C}\Phi_k(q)=\frac{1}{(q;q)_{\infty}^{k}}\sum\limits_{m_1,\cdots ,m_{k-1}\in \mathbf{Z}}q^{Q(m_1,\cdots, m_{k-1})},
\end{equation}
where
\begin{equation}\label{Q}Q(m_1,m_2,\cdots, m_{k-1})=\sum\limits_{i=1}^{k-1}m_{i}^{2}+\sum\limits_{1\le i<j\le k-1}m_{i}m_{j}\end{equation}
and
\[(a;q)_\infty = \prod_{j=1}^\infty \left(1-aq^{j-1}\right), \quad |q|<1.\]
Using (\ref{notation}), Andrews \cite[Corollary 5.2]{Andrews} discovered alternative expressions for $\mathrm{C}\Phi_{k}(q)$ when $k=2,3$ and $5$.
To describe Andrews' identities, let $q=e^{2\pi i \tau}$  throughout this paper,
\begin{align*}
\Theta_3(q)=\vartheta_3(0|2\tau)=\sum_{j=-\infty}^{\infty} q^{j^2}, \,\text{and}\, \Theta_2(q)=\vartheta_2(0|2\tau)=\sum_{j=-\infty}^{\infty}q^{(j+1/2)^2},
\end{align*}
where
$$\vartheta_2(u|\tau) = \sum_{j=-\infty}^\infty e^{\pi  i \tau(j+1/2)^2}e^{(2j+1)iu}$$
and
$$\vartheta_3(u|\tau) = \sum_{j=-\infty}^\infty e^{\pi i \tau j^2}e^{2ji u}.$$

Andrews showed that
\begin{align}\label{c-phi-2}
\mathrm{C}\Phi_{2}(q)&=\frac{(q^2;q^4)_{\infty}}{(q;q^2)_{\infty}^{4}(q^4;q^4)_{\infty}},\\
 \label{c-phi-3}
\mathrm{C}\Phi_{3}(q)&=\frac{1}{(q;q)_{\infty}^{3}}\left(\Theta_3(q)\Theta_3(q^3)+\Theta_2(q)\Theta_2(q^3) \right)\\
&=\frac{1}{(q;q)_\infty^3}\left(1+6\sum_{j=0}^\infty \left(\frac{j}{3}\right)\frac{q^j}{1-q^j}\right)\label{c-phi-3-2}
\intertext{and}
\mathrm{C}\Phi_{5}(q)&=\frac{1}{(q;q)_\infty^5}\left(1+25\sum_{j=1}^\infty \left(\frac{j}{5}\right)\frac{q^j}{(1-q^j)^2}-5\sum_{j=1}^\infty \left(\frac{j}{5}\right) \frac{jq^j}{1-q^j}\right)   \label{c-phi-5-2}
\end{align}
where $\left(\dfrac{j}{\cdot}\right)$ is the Kronecker symbol.
For \eqref{c-phi-5-2}, we have recorded the equivalent version of Andrews' identity found in the work of L.W. Kolitsch \cite[Lemma 1]{Kolitsch1989}.
Andrews \cite[pp. 13--15]{Andrews} used Jacobi triple product identity (see for example \cite[(3.1)]{Andrews}) and properties of theta series to prove \eqref{c-phi-2} and \eqref{c-phi-3}. The proofs of  \eqref{c-phi-3-2} and \eqref{c-phi-5-2}
\cite[pp. 26--27]{Andrews} are dependent on the work of H.D. Kloosterman \cite[p. 362, p. 358]{Kloosterman}.
In a paragraph before the proofs of \eqref{c-phi-3-2} and \eqref{c-phi-5-2}, Andrews \cite[p. 26]{Andrews} mentioned that similar identity exists for $k=7$, but this identity was not given in \cite{Andrews}. This missing identity, namely,
\begin{equation}\label{c-phi-7-Ei}
\mathrm{C}\Phi_{7}(q)=\frac{1}{(q;q)_{\infty}^{7}}\left(1+\frac{343}{8}\sum_{j=1}^\infty \left(\frac{j}{7}\right)
\frac{q^j+q^{2j}}{(1-q^j)^3}-\frac{7}{8}\sum_{j=1}^\infty \left(\frac{j}{7}\right)\frac{j^2q^j}{1-q^j}\right),
\end{equation}
was later published by Kolitsch \cite[Lemma 2]{Kolitsch1989}.

Recently,  N.D. Baruah and B.K. Sarmah \cite{BaruahDisc,Baruah-Rama} used the method illustrated in
 Z. Cao's work \cite{Cao} and  found representations of $\mathrm{C}\Phi_{k}(q)$ for $k=4,5$ and $6$. They showed that
\begin{equation}\label{c-phi-4}
\mathrm{C}\Phi_{4}(q)=\frac{1}{(q;q)_{\infty}^{4}}\left(\Theta_3^3(q^2)+3\Theta_3(q^2)\Theta_2^2(q^2) \right),
\end{equation}
\begin{align}\label{c-phi-5}
\mathrm{C}\Phi_{5}(q)=\frac{1}{(q;q)_{\infty}^{5}}\Big(&\Theta_3(q^{10})\Theta_3^{3}(q^2)+3\Theta_3(q^{10})\Theta_3(q^2)\Theta_2^2(q^2)+\frac{1}{2}
\Theta_2(q^{5/2})\Theta_2^3(q^{1/2})\notag \\
&+3\Theta_2(q^{10})\Theta_2(q^2)\Theta_3^{2}(q)+\Theta_2(q^{10})\Theta_2^{3}(q^2)  \Big)
\end{align}
and
\begin{align}\label{c-phi-6}
\mathrm{C}\Phi_{6}(q)=&\frac{1}{(q;q)_{\infty}^{6}}\Big(
\Theta_3^3(q)\Theta_3(q^2)\Theta_3(q^6)+\frac{3}{4}\Theta_2^3(q^{1/2})\Theta_2(q)\Theta_2(q^{3/2})\nonumber\\
&\quad +\Theta_3^2(q)\Theta_2(q^2)\Theta_2(q^6)\Big).
\end{align}
Identities \eqref{c-phi-4} and \eqref{c-phi-5} can be found in \cite[(2.2)]{BaruahDisc} and \cite[(2.13)]{BaruahDisc} respectively while
\eqref{c-phi-6} can be found in \cite[(2.1)]{Baruah-Rama}.

For $k>7$, it is not clear if new identities associated with $\mathrm{C}\Phi_{k}(q)$ could be derived using
the methods of Andrews and Baruah-Sarmah.
In fact, Andrews \cite[p.\ 15]{Andrews} commented that as $k$ increases, ``the expressions quickly become long and messy''. The main goal of this paper is to discuss ways of finding new representations of $\mathrm{C}\Phi_{k}(q)$. Using the theory of modular forms, we will derive all the identities mentioned above. In addition to providing new proofs to known identities, we will also construct
new representations for $\mathrm{C}\Phi_{k}(q)$ for the first time for $8\leq k\leq 17$.
In Section 2, we discuss the behavior of $\mathrm{C}\Phi_k(q)$ as modular form for each integer $k>2$.
In Section 3, we derive alternative representations of
$\mathrm{C}\Phi_k(q)$ for primes $k=3,5,7,11,13$ and $17$ and prove
Kolitsch's identities \cite[p. 223]{Kolitsch1989}
\begin{align}\label{Kolitsch5}
c\phi_5(n) &=p(n/5)+5 p(5n-1)\intertext{and}
c\phi_7(n) &= p(n/7)+7p(7n-2).\label{Kolitsch7}
\end{align}
We also
discover and prove the identities
\begin{align}\label{Analogue11}
c\phi_{11}(n) &= p(n/11)+11p(11n-5), \intertext{and}
c\phi_{13}(n) &= p(n/13)+13 p(13n-7) + 26 a(n)\label{Analogue13}
\end{align}
where $p(x)=0$ when $x$ is not an integer and
$$q\prod_{j=1}^\infty \frac{(1-q^{13j})}{(1-q^j)^2} = \sum_{j=0}^\infty a(j) q^j.$$
It turns out that \eqref{Analogue11} is equivalent to Kolitsch's identity for 11-colored generalized Frobenius partition with order 11
\cite[Theorem 3]{Kolitsch1991} which was
first established using the results of F.G. Garvan, D. Kim and D. Stanton \cite{Garvan-Kim-Stanton}. Identity \eqref{Analogue13}, on the other hand, is new.
The proof of \eqref{Analogue13} motivates the discovery of a uniform method in treating identities
such as \eqref{Analogue13}. We discuss this method in Section 4 and derive analogues of
 \eqref{Analogue13} for $\ell=17, 19$ and 23. This method also leads to the discovery of interesting modular functions
that satisfy mysterious congruences. For example, if
$$h_\ell(\tau) = (q^\ell;q^\ell)_\infty\mathrm{C}\Phi_\ell(q)-
1-\ell (q^\ell;q^\ell)_\infty\sum_{j=1}^\infty p\left(\ell j -\frac{\ell^2-1}{24}\right)q^j
-2\ell^{(\ell-11)/2}\frac{\eta^{\ell-11}(\ell\tau)}{\eta^{\ell-11}(\tau)},$$
where $\eta(\tau)$ is the Dedekind eta function given by
\[\eta(\tau)=q^{\frac{1}{24}}\prod_{n=1}^{\infty}(1-q^n),\]
then for $\ell=17,19$ and 23,
$$h_\ell(\tau) \equiv 0 \,\,\pmod{\nu_\ell}$$
where $$\nu_\ell=\ell^2-\ell p\left(\ell n -\frac{\ell^2-1}{24}\right).$$

In Section 5, we discuss the cases for  $k=9$ and 15, the two composite odd integers less than 17.
We derive the following  congruence satisfied by $c\phi_k(n)$:
\begin{equation}\label{main-cong-power} c\phi_{p^\alpha N}(n)\equiv c\phi_{p^{\alpha-1}N}(n/p)\pmod{p^{2\alpha}},\end{equation}
where $c\phi_k(m)=0$ if $m$ is not an integer, $p$ a prime, $N$ and $\alpha$ are positive integers with $(N,p)=1$.
The discovery of congruence \eqref{main-cong-power} is motivated
 by congruences found in the study of $\mathrm{C}\Phi_{10}(q)$ and $\mathrm{C}\Phi_{14}(q)$ in Section 6 where identities associated with
 $k=4,6,8,10,12$ and $16$ are given.
More precisely, we discovered that
\begin{equation}\label{2pto2}\mathrm{C}\Phi_{2p}(q)\equiv \frac{\Theta_3(q^p)}{(q^p;q^p)_\infty}=\mathrm{C}\Phi_{2}(q^p)\pmod{p^2},\end{equation}
which holds for any odd prime $p$.
The second equality follows from Andrews' identity for $\mathrm{C}\Phi_{2}(q)$ (see also \eqref{c-phi-2-mod}).
Congruence \eqref{2pto2} can be viewed as an extension of Andrews' congruence \cite[Corollary 10.2]{Andrews}
\begin{equation}\label{andrews-cphip}\mathrm{C}\Phi_{p}(q)\equiv \frac{1}{(q^p;q^p)_\infty} \pmod{p^2}\end{equation}
if we rewrite \eqref{andrews-cphip}
as
\begin{equation}\label{andrews-cphip-2}\mathrm{C}\Phi_{p}(q)\equiv \mathrm{C}\Phi_{1}(q^p) \pmod{p^2}\end{equation}
using the fact that $$c\phi_1(n) =p(n).$$
The discovery of \eqref{2pto2} leads
 to the congruence
\begin{equation}\label{pltol}\mathrm{C}\Phi_{\ell p}(q)\equiv  \mathrm{C}\Phi_{\ell}(q^p)\pmod{p^2},\end{equation}
which holds for any distinct primes $\ell$ and $p$. Congruence \eqref{pltol} eventually  leads to \eqref{main-cong-power}.

There may be more surprising properties to be discovered for $c\phi_k(n)$  and we hope that this article will be helpful to future researchers who are interested in knowing more about these functions.

\section{Modular properties of $\mathrm{C}\Phi_k(q)$}

In this section, we determine the modular properties of the function $$\mathfrak{A}_{k}(q):=(q;q)_{\infty}^{k}\mathrm{C}\Phi_{k}(q)=\sum\limits_{m_1,\cdots ,m_{k-1}\in
\mathbf{Z}}q^{Q(m_1,\cdots, m_{k-1})}, \quad k>1.$$

Let $\chi$ be a Dirichlet character and $M_{k}(\Gamma_{0}(N),\chi)$  be the space of modular forms on $\Gamma_{0}(N)$ with weight $k$ and multiplier $\chi$. When $\chi$ is the trivial Dirichlet character, we  write $M_{k}(\Gamma_{0}(N))$ for $M_{k}(\Gamma_{0}(N),\chi)$.

Let
\begin{equation} \label{An}
A_n=\left(\begin{matrix}
2 & 1 & 1 & \cdots & 1 \\
1 &2 & 1& \cdots & 1 \\
\cdots & \cdots &\cdots & \cdots & \cdots\\
1 & 1 & 1& \cdots & 2
\end{matrix}\right)_{n\times n}\, .
\end{equation}
Then $\det(A_n)=n+1$ and
\begin{equation} \label{An-1}
A_n^{-1}=\left(\begin{matrix}
\dfrac{n}{n+1} & -\dfrac{1}{n+1} &  -\dfrac{1}{n+1} & \cdots &  -\dfrac{1}{n+1} \\
& & & & \\
 -\dfrac{1}{n+1} & \dfrac{n}{n+1} &  -\dfrac{1}{n+1} & \cdots &  -\dfrac{1}{n+1}\\
 & & & & \\
 \cdots & \cdots & \cdots & \cdots & \cdots &\\
 & & & & \\
  -\dfrac{1}{n+1} &  -\dfrac{1}{n+1} &  -\dfrac{1}{n+1} & \cdots & \dfrac{n}{n+1}
\end{matrix}\right).
\end{equation}

Let $n$ be a positive even integer and
\[\chi(\cdot) =\left(\frac{(-1)^{n/2}\det(A_n)}{\cdot} \right) =\left(\frac{(-1)^{n/2}(n+1)}{\cdot} \right).\]
Since all the diagonal components of $A_{n}$ and $(n+1)A_{n}^{-1}$ are even,
we deduce from \cite[Corollary 4.9.5 (3)]{Miyake} that if
\[\theta(\tau;A_n)=\sum\limits_{m\in \mathbf{Z}^{n}}e^{\pi i \tau\cdot m^t A_n m}=\sum\limits_{m\in \mathbf{Z}^{n}}q^{\frac{1}{2}m^t A_n m}=\sum\limits_{m\in \mathbf{Z}^{n}}q^{Q(m_1,\cdots m_{n})},\] then
%By \cite[Corollay 4.9.5 (3)]{Miyake} where we set $P(m)=1$,  we deduce that
\begin{equation}\label{odd-case}
\theta(\tau;A_n)=\sum\limits_{m\in \mathbf{Z}^{n}}q^{Q(m_1,\cdots m_{n})} = \mathfrak{A}_{n+1}\in M_{n/2}\big(\Gamma_{0}(n+1),\chi \big).
\end{equation}

Next, let $n>1$ be an odd positive integer and
$$B_n=\left(\begin{matrix} A_{n} & 0 \\ 0& 2 \end{matrix}\right).$$
Then $\det B_n=2(n+1)$.
We have
\begin{eqnarray*}
\theta(\tau;B_n) &=&\sum\limits_{m\in \mathbf{Z}^{n+1}}e^{\pi i \tau\cdot m^t B_n m}=\sum\limits_{m\in \mathbf{Z}^{n}}q^{\frac{1}{2}m^t A_{n} m}\sum\limits_{m_{n+1}\in \mathbf{Z}}q^{m_{n+1}^2}\\
&=&\sum\limits_{m\in \mathbf{Z}^{n}}q^{Q(m_1,\cdots m_{n})}\sum\limits_{m_{n+1}\in \mathbf{Z}}q^{m_{n+1}^2}\\
&=&\mathfrak{A}_{n+1}(q)\Theta_3(q).
\end{eqnarray*}
Note that
\[B_n^{-1}=\left(\begin{matrix} A_n^{-1} & 0 \\ 0 & \frac{1}{2} \end{matrix} \right).\]
Let
$$\chi(\cdot) =\left(\frac{(-1)^{(n+1)/2}\det(B_n)}{\cdot} \right) =\left(\frac{2(-1)^{(n+1)/2}(n+1)}{\cdot} \right).$$
Since all the diagonal components of $B_{n}$ and $2(n+1)B_{n}^{-1}$ are even, we deduce from \cite[Corollary 4.9.5 (3)]{Miyake} that
\begin{equation}\label{even-case}
\theta(\tau;B_n)=\mathfrak{A}_{n+1}(q)\Theta_3(q) \in M_{(n+1)/2}\big(\Gamma_{0}(2(n+1)),\chi \big).
\end{equation}

Similarly,  let
\[C_n=\left(\begin{matrix} A_{n} & 0 \\ 0& 4 \end{matrix}\right).  \]
Then $\det C_n=4(n+1)$.
 Note that
\[C_n^{-1}=\left(\begin{matrix} A_{n}^{-1} & 0 \\ 0 & \frac{1}{4} \end{matrix} \right).\]
Let
\[\chi(\cdot) =\left(\frac{(-1)^{(n+1)/2}\det(C_n)}{\cdot} \right) =\left(\frac{(-1)^{(n+1)/2}(n+1)}{\cdot} \right).\]
Since all the diagonal components of $C_{n}$ and $4(n+1)C_{n}^{-1}$ are even, we find from  \cite[Corollay 4.9.5 (3)]{Miyake} that
\begin{equation}\label{even-case-2}
\theta(\tau;C_n)=\mathfrak{A}_{n+1}(q)\Theta_3(q^2) \in M_{(n+1)/2}\big(\Gamma_{0}(4(n+1)),\chi \big).
\end{equation}

From \eqref{odd-case}, \eqref{even-case} and \eqref{even-case-2}, we deduce the following theorem:
\begin{theorem}\label{generating-thm}
If $k=2r+1$ is odd, then
\[\mathfrak{A}_k(q) \in M_{(k-1)/2}\Big(\Gamma_{0}(k), \big(\frac{(-1)^{r}\cdot k}{\cdot} \big) \Big).\]
If $k=2r$ is even, then
\[\Theta_3(q)\mathfrak{A}_k(q) \in M_{k/2}\Big(\Gamma_{0}(2k), \big(\frac{2(-1)^{r}\cdot k}{\cdot} \big) \Big),\]
and
\[\Theta_3(q^2)\mathfrak{A}_k(q) \in M_{k/2}\Big(\Gamma_{0}(4k), \big(\frac{(-1)^{r}\cdot k}{\cdot} \big) \Big).\]
\end{theorem}

\section{Generating function of $c\phi_{k}(n)$ when $k$ is a prime}

In this section, we will derive expressions for ${\rm C}\Phi_k(q)$ when $k$ is a prime number less than 18.

\subsection{Case $k=2$}\quad
\medskip

Our proof for $k=2$ is exactly the same as that of  Andrews' proof of \eqref{c-phi-2} and we include it for the sake of completeness.
From \eqref{notation}, we find that
\begin{equation}\label{c-phi-2-mod}
\mathrm{C}\Phi_{2}(q) = \frac{\Theta_3(q)}{(q;q)_\infty^2}.
\end{equation}
Using Jacobi triple product identity (see \cite[(3.1)]{Andrews}), we deduce that
\begin{equation}\label{T3}\Theta_3(q) = (-q;q^2)_\infty^2(q^2;q^2)_\infty.\end{equation}
Substituting \eqref{T3} into \eqref{c-phi-2-mod} and simplifying, we complete the proof of \eqref{c-phi-2}.

\subsection{Case $k=3$}\quad
\medskip

From Theorem \ref{generating-thm}, we deduce that
$\mathfrak{A}_3(q)$ is a modular form of weight 1 on $\Gamma_0(3)$ with multiplier $\left(\frac{-3}{\cdot}\right)$.
Comparing the coefficients of $\mathfrak{A}_3(q)$ with the known Eisenstein series of weight 1 \cite[Theorem 4.8.1]{Diamond}
on $\Gamma_0(3)$ with multiplier $\left(\frac{-3}{\cdot}\right)$, we deduce that
$$\mathfrak{A}_3(q) =\sum_{m,n=-\infty}^\infty q^{m^2+mn+n^2}=
1+6\sum_{j=1}^\infty \left(\frac{3}{j}\right) \dfrac{q^j}{1-q^j}.$$ This is equivalent to \eqref{c-phi-3-2}.
Another proof of \eqref{c-phi-3-2} can also be found, for example,
in the article by J.M. Borwein, P.B. Borwein and F.G. Garvan \cite[p. 43]{Borwein}.

We next show that \eqref{c-phi-3} follows from a general identity.
Let $\omega=(1+\sqrt{-d})/2$, with $d\equiv 3\pmod{4}.$ Observe that
the set
$$S=\left\{m+n\omega|m,n\in\mathbf Z\right\}$$ is a disjoint union of
$$S_0=\left\{m+n\omega|m,n\in\mathbf Z,n\equiv 0\!\!\!\!\pmod{2}\right\}
\,\,\text{and}\,\, S_1=\left\{m+n\omega|m,n\in\mathbf Z, n\equiv 1\!\!\!\!\pmod{2}\right\}.$$
Let $$N(m+n\omega)=m^2+mn+\left(\frac{d+1}{4}\right)n^2.$$
 Then
\begin{equation*}\sum_{v\in S}q^{N(v)} = \sum_{v\in S_0}q^{N(v)}+\sum_{v\in S_1}q^{N(v)}.\end{equation*}
Simplifying the above, we deduce that
{
\begin{equation}\label{id}
\sum_{m,n\in\mathbf Z} q^{m^2+mn+\left(\frac{d+1}{4}\right)n^2} = \Theta_3(q)\Theta_3(q^d)+\Theta_2(q)\Theta_2(q^d).
\end{equation}
}
Identity \eqref{c-phi-3} follows from \eqref{id} with $d=3$.

\subsection{Case $k=5$}\label{sub5}\quad
\medskip

We first establish three representations of $\mathrm{C}\Phi_{5}(q)$:
\begin{theorem}\label{cphi5-thm}
The following identities hold:
\begin{align}\label{Phi-5-E}
\mathrm{C}\Phi_{5}(q)&=\frac{1}{(q;q)_\infty^{5}}\left(1+25\sum_{j=1}^\infty \left(\frac{j}{5}\right)\frac{q^j}{(1-q^j)^2}-5\sum_{j=1}^\infty \left(\frac{j}{5}\right) \frac{jq^j}{1-q^j}\right)
\\
\label{c-phi-5-product}
&=\frac{1}{(q^5;q^5)_\infty}+25q\frac{(q^5;q^5)_\infty^{5}}{(q;q)_\infty^{6}}\\
&=\frac{1}{(q^5;q^5)_\infty} + 5\sum_{j=1}^\infty p(5j-1)q^j.\label{c-phi-5-right}
\end{align}
\end{theorem}

\begin{proof}
From Theorem \ref{generating-thm}, we deduce that
$$\mathfrak{A}_{5}(q) \in M_{2}\left(\Gamma_{0}(5), \left(\frac{5}{\cdot} \right) \right).$$
Since \cite[Theorem 1.34]{Ono}
$$\dim M_{2}\left(\Gamma_{0}(5), \left(\frac{5}{\cdot} \right)\right)=2,$$
we deduce that the two modular forms
\begin{align}\label{Eisenstein-basis}
E_{5,1}&=\sum_{m=1}^\infty\sum_{d=1}^\infty \left(\frac{d}{5}\right)m q^{md}=\sum_{j=1}^\infty \left(\frac{j}{5}\right)\frac{q^j}{(1-q^j)^2}
\intertext{and}
E_{5,2}&= 1-5\sum_{m=1}^\infty\sum_{d=1}^\infty \left(\frac{d}{5}\right)dq^{md}=1-5\sum_{j=1}^\infty \left(\frac{j}{5}\right) \frac{jq^j}{1-q^j},\notag
\end{align} which are in $M_2\left(\Gamma_{0}(5), \left(\frac{5}{\cdot} \right) \right)$ (see \cite[Sec.\ 4.6]{Diamond}), form a basis for this space of modular forms.
By comparing Fourier coefficients of $\mathfrak{A}_5(q), E_{5,1}$ and $E_{5,2}$, we deduce that
\[\mathfrak{A}_{5}(q)=25E_{5,1}+E_{5,2}\]
and the proof of \eqref{Phi-5-E} is complete.

Before we begin with our proof of  \eqref{c-phi-5-product}, we observe that if $p>3$, then
by
Theorem \ref{generating-thm},
$${\rm C}\Phi_p(q) (q^p;q^p)_\infty$$ is a modular function on $\Gamma_0(p).$
This implies that the function can be expressed in terms of combinations of infinite
products.
For more details,
see for example the paper by H.H.Chan, H. Hahn, R.P. Lewis and S.L. Tan \cite{MRL}.
In \cite[Corollary 10.2]{Andrews}, Andrews showed that
if $p$ is a prime, then
$${\rm C}\Phi_p(q) = \frac{1}{(q^p;q^p)_\infty} + p^2 G_p(q)$$
for some $G_p(q)$ analytic inside $|q|<1$ with integral power
series coefficients. He then asked \cite[Problem 6]{Andrews}
for explicit closed forms for $G_p(q)$.
Since $$G_p(q)(q^p;q^p)_\infty = \frac{1}{p^2}\left({\rm C}\Phi_p(q)
(q^p;q^p)_\infty-1\right),
$$
we conclude that $G_p(q)$
is a modular function on $\Gamma_0(p)$ for $p>3$. This
provides an answer to Andrews' question.
The above discussion also gives us a way to derive alternative
expressions for ${\rm C}\Phi_p(q)$ whenever the functions invariant under
 $\Gamma_0(p)$ can be expressed as a rational function of a single modular function.
 This happens for $p=5,7,$ and $13$.
We now use this fact to derive an expression for ${\rm C}\Phi_5(q).$
It is known from T. Kondo's work \cite{Kondo} that every modular function on $\Gamma_0(5)$ is
a rational function of
${\eta^6(5\tau)}/{\eta^6(\tau)},$
where $$\eta(\tau) = e^{\pi i \tau/12}\prod_{j=1}^\infty \left(1-e^{2\pi i j\tau}\right).$$
Since
${\rm C}\Phi_5(q) (q^5;q^5)_\infty$ is a modular function on $\Gamma_0(5)$, we deduce that
$${\rm C}\Phi_5(q) (q^5;q^5)_\infty=1+25\frac{\eta^6(5\tau)}{\eta^6(\tau)}.$$
This completes the proof of \eqref{c-phi-5-product}.

Using the fact that
$$\frac{1}{(q;q)_\infty}=\sum_{j=0}^\infty p(j)q^j$$
and Ramanujan's identity \cite[Theorem 2.3.4]{Berndt},
\begin{equation}\label{Ramapart5}\sum_{j=1}^\infty p(5j-1)q^j = 5q\frac{(q^5;q^5)_\infty^5}{(q;q)_\infty^6},\end{equation}
we deduce \eqref{c-phi-5-right} from \eqref{c-phi-5-product}.
\end{proof}
\begin{rem}
Identity \eqref{Phi-5-E} is Andrews' \eqref{c-phi-5-2}, which was first proved using results found in Kloosterman's work \cite{Kloosterman}.
Identity  \eqref{c-phi-5-right}   immediately implies
\eqref{Kolitsch5}.
We emphasize here that our proof of \eqref{Kolitsch5} is different from Kolitsch's proof as we have used
\eqref{c-phi-5-product} instead of \eqref{c-phi-5-2}.
\end{rem}

As shown in  \eqref{c-phi-5}, there is a fourth representation  of $\mathrm{C}\Phi_{5}(q)$ due to Baruah and Sarmah.
This identity can be proved by realizing that
$$\mathfrak{A}_5(q) \in M_{2}\left(\Gamma_{0}(40), \left(\frac{5}{\cdot} \right) \right),$$
together with the fact that the space
$M_{2}\Big(\Gamma_{0}(40), \big(\frac{5}{\cdot} \big) \Big)$ is spanned by
the modular forms
\begin{alignat*}{3}
&\Theta_3(q)\Theta_3^{3}(q^5),  &&\Theta_3^{3}(q)\Theta_3(q^5),    && \Theta_3(q^2)\Theta_3^{3}(q^{10}),\quad \Theta_3^{3}(q^2)\Theta_3(q^{10}), \\ &\Theta_3(q)\Theta_3(q^5)\Theta_2^2(q^2),  && \Theta_3(q^2)\Theta_3(q^{10})\Theta_2^{2}(q^2), && \Theta_2^{3}(q^{1/2})\Theta_2(q^{5/2}), \quad\\ &\Theta_3^{2}(q)\Theta_2(q^2)\Theta_2(q^{10}), \quad
&& \Theta_2^{3}(q^2)\Theta_2(q^{10}), \quad \text{and}\quad  &&\Theta_3^{2}(q^5) \Theta_2(q^2)\Theta_2(q^{10}).
\end{alignat*}

\subsection{Case $k=7$}
\quad
\medskip

\begin{theorem}\label{cphi7-thm}
The following identities are true:
\begin{align}
\mathrm{C}\Phi_{7}(q)&=\frac{1}{(q;q)_{\infty}^{7}}\left(1-\frac{7}{8}\sum_{k=1}^\infty \left(\frac{k}{7}\right)\frac{k^2q^k}{1-q^k}+\frac{343}{8}\sum_{k=1}^\infty \left(\frac{k}{7}\right)
\frac{q^k+q^{2k}}{(1-q^k)^3}\right) \label{Phi-7-E} \\
&=\frac{1}{(q^7;q^7)_\infty}+49q\frac{(q^7;q^7)_\infty^{3}}{(q;q)_\infty^{4}}+343q^2\frac{(q^7;q^7)_\infty^{7}}{(q;q)_\infty^{8}} \label{c-phi-7-product}\\
&=\frac{1}{(q^7;q^7)_\infty} + 7\sum_{j=1}^\infty p(7j-2)q^j. \label{c-phi-7-right}
\end{align}
\end{theorem}
\begin{proof}
Before giving the proof of \eqref{Phi-7-E}, we observe that \eqref{Phi-7-E} is the same as \eqref{c-phi-7-Ei}.
We will prove \eqref{Phi-7-E} using the theory of modular forms. Note that by Theorem \ref{generating-thm}, we have  $\mathfrak A_{7}(q) \in M_{3}\left(\Gamma_{0}(7), \left(\frac{-7}{\cdot} \right)\right)$. The space $M_{3}\left(\Gamma_{0}(7), \left(\frac{-7}{\cdot} \right)\right)$ is spanned by
\[E_{7,1}=\sum_{m=1}^\infty\sum_{d=1}^\infty \left(\frac{d}{7}\right)m^2q^{md}=\sum_{j=1}^\infty \left(\frac{j}{7}\right)
\frac{q^j+q^{2j}}{(1-q^j)^3},\]
\[E_{7,2}=1-\frac{7}{8}\sum_{m=1}^\infty\sum_{d=1}^\infty \left(\frac{d}{7}\right)d^2q^{md}=1-\frac{7}{8}\sum_{j=1}^\infty \left(\frac{j}{7}\right)\frac{j^2q^j}{1-q^j} ,\]
and
\[S_{7}=\eta^3(\tau)\eta^3(7\tau).\]
By comparing Fourier coefficients of these modular forms,
we deduce that
\[\mathfrak{A}_7(q)=\frac{343}{8}E_{7,1}+E_{7,2}.\]
This completes the proof of \eqref{Phi-7-E}.

The proof of \eqref{c-phi-7-product} is similar to the proof of \eqref{c-phi-5-product}.
We recall that modular functions
invariant under $\Gamma_0(7)$ is a rational function of
$$\frac{\eta^4(7\tau)}{\eta^4(\tau)}.$$
Since $(q^7;q^7)_\infty \mathrm{C}\Phi_7(q)$ is such a function, we conclude that
$$(q^7;q^7)_\infty \mathrm{C}\Phi_7(q)=1+49\frac{\eta^4(7\tau)}{\eta^4(\tau)}+343\frac{\eta^8(7\tau)}{\eta^8(\tau)}$$ and the proof
of \eqref{c-phi-7-product} is complete.

Ramanujan discovered that \cite[Theorem 2.4.2]{Berndt}
\begin{equation}\label{Ramapart7}\sum_{j=1}^\infty p(7j-2)q^j = 7q \frac{(q^7;q^7)_\infty^3}{(q;q)_\infty^4}+
49 q^2 \frac{(q^7;q^7)_\infty^7}{(q;q)_\infty^8}.\end{equation} Using \eqref{Ramapart7} and \eqref{c-phi-7-product},
we deduce
\eqref{c-phi-7-right}.

\end{proof}

Identity \eqref{c-phi-7-right}
immediately implies
Kolitsch's identity
\eqref{Kolitsch7}. We emphasize here that our proof of \eqref{Kolitsch7} uses \eqref{c-phi-7-product} instead of
\eqref{Phi-7-E}.

As in the case for $k=5$, we are able to find a representation of $\mathrm{C}\Phi_{7}(q)$ in terms of theta functions. This new identity
is an analogue of \eqref{c-phi-5}. We first observe that $\mathfrak{A}_7(q)\in M_{3}\left(\Gamma_{0}(28), \left(\frac{-28}{\cdot} \right)  \right)$. Furthermore
the modular forms
\begin{alignat*}{3}
&\Theta_3^{5}(q)\Theta_3(q^7), \quad && \Theta_3^{3}(q)\Theta_3^{3}(q^7), \quad && \Theta_3(q)\Theta_3^{5}(q^7), \quad \Theta_3^{4}(q)\Theta_2(q^{1/2})\Theta_2(q^{7/2}), \\
& \Theta_3^{4}(q^7)\Theta_2(q^{1/2})\Theta_2(q^{7/2}), \quad && \Theta_3^{4}(q)\Theta_2(q)\Theta_2(q^{7}), \quad
&& \Theta_3^{4}(q^7)\Theta_2(q)\Theta_2(q^{7}), \quad  \\ &\Theta_2^3(q^{1/2})\Theta_2^3(q^{7/2}),
&& \Theta_2^3(q)\Theta_2^3(q^{7}),\quad && \Theta_2^5(q)\Theta_2(q^{7}), \quad  \text{and}\quad \Theta_2(q)\Theta_2^5(q^{7})
\end{alignat*}
form a basis for $M_{3}\left(\Gamma_{0}(28), \left(\frac{-28}{\cdot} \right)  \right)$. Hence, we deduce that
\begin{align}\label{c-phi-7-theta}\mathrm{C}\Phi_{7}(q)&=\frac{1}{(q;q)_{\infty}^7}\Bigg(-\frac{15}{32}\Theta_3^{5}(q)\Theta_3(q^7)+\frac{55}{16}\Theta_3^{3}(q)\Theta_3^{3}(q^7)-\frac{63}{32}\Theta_3(q)\Theta_3^{5}(q^7) \notag\\
&\qquad\qquad\qquad+\frac{15}{16}\Theta_3^{4}(q)\Theta_2(q^{1/2})\Theta_2(q^{7/2})
+\frac{105}{16}\Theta_3^{4}(q^7)\Theta_2(q^{1/2})\Theta_2(q^{7/2}) \notag\\
&\qquad\qquad\qquad-\frac{15}{16}\Theta_3^{4}(q)\Theta_2(q)\Theta_2(q^7)
+\frac{525}{16}\Theta_3^{4}(q^7)\Theta_2(q)\Theta_2(q^7) \notag \\
&\qquad\qquad\qquad+\frac{105}{32}\Theta_2^3(q^{1/2})\Theta_2^3(q^{7/2}) +\frac{95}{8}\Theta_2^3(q)\Theta_2^3(q^7)  \notag\\
&\qquad\qquad\qquad+\frac{15}{16} \Theta_2^5(q)\Theta_2(q^7) -\frac{189}{16}\Theta_2(q)\Theta_2^5(q^{7}) \Bigg).\end{align}

We next prove some congruences satisfied by $c\phi_{7}(n)$ using \eqref{c-phi-7-theta} and \eqref{c-phi-7-product}.

\begin{theorem}\label{c-phi-7-mod-5}
For any integer $n\geq 0$,
\[c\phi_{7}(5n+3) \equiv 0 \pmod{5}.\]
\end{theorem}

\begin{proof}
From \eqref{c-phi-7-theta}, we deduce that
\begin{eqnarray}\label{c-phi-7-mod-5-start}
\sum_{j=0}^{\infty}c\phi_{7}(j)q^j &\equiv& \frac{1}{(q;q)_{\infty}^{7}}\Big(\Theta_3(q)\Theta_3^{5}(q^7)+\Theta_2(q)\Theta_2^5(q^7) \Big) \pmod{5} \nonumber \\
&\equiv & \frac{1}{(q^5;q^5)_{\infty}^{2}}\left((q;q)_{\infty}^{3}\Theta_3(q)\Theta_3(q^{35})+(q;q)_{\infty}^{3}
\Theta_2(q)\Theta_2(q^{35})\right)  \pmod{5}. \nonumber \\
\end{eqnarray}
Using Jacobi's identity for $(q;q)_\infty^3$ \cite[Theorem 1.3.9]{Berndt}, we find that
\begin{equation}\label{c-phi-7-cong-expan-1}
(q;q)_{\infty}^{3}\Theta_3(q)=\sum_{i=0}^{\infty}\sum_{j=-\infty}^{\infty}(-1)^{i}(2i+1)q^{i(i+1)/2+j^2}.
\end{equation}
Now, observe that
$$m=\frac{i(i+1)}{2}+j^2$$  is equivalent to $$8m+1=(2i+1)^2+8j^2.$$
If $8m\equiv -1 \pmod{5}$, then $m\equiv 3\pmod{5}$.
Since
$$\left(\dfrac{-8}{5} \right)=-1,$$ we deduce that
$$(2i+1)^2+8j^2\equiv 0\pmod{5}$$ holds if and only if
$$2i+1\equiv j \equiv 0\pmod{5}.$$

Similarly, we have
\begin{equation}\label{c-phi-7-cong-expan-2}
q^{35/4}(q;q)_{\infty}^{3}\Theta_2(q)=\sum\limits_{i=0}^{\infty}\sum\limits_{j=0}^{\infty}(-1)^{i}(2i+1)q^{9+i(i+1)/2+j(j+1)}.
\end{equation}
Observe that
\[m=9+\frac{i(i+1)}{2}+j(j+1)\] is equivalent to
\[8m-69=(2i+1)^2+2(2j+1)^2.\]
Note that if $8m-69\equiv 0\pmod{5}$ then $m\equiv 3\pmod{5}$.
Since $$\left(\frac{-2}{5} \right)=-1,$$ we deduce that  $$(2i+1)^2+2(2j+1)^2\equiv 0\pmod{5}$$ holds if
and only if  $$2i+1\equiv 2j+1 \equiv 0\pmod{5}.$$

From (\ref{c-phi-7-mod-5-start}), (\ref{c-phi-7-cong-expan-1}) and (\ref{c-phi-7-cong-expan-2}), we conclude that if $m\equiv 3\pmod{5}$ then $$c\phi_{7}(m)\equiv 0\pmod{5},$$
or equivalently,
$$ c\phi_7(5n+3)\equiv 0\pmod{5}$$
for any integer $n\geq 0$.
\end{proof}
\begin{rem} It is possible to deduce Theorem \ref{c-phi-7-mod-5} without using \eqref{c-phi-7-theta}.
We first recall a recent result of
F.G. Garvan and J.A.
Sellers \cite{Garvan} which states that
if  $p$ is a prime number and $0<r<p$, then the congruence
$$c\phi_k(pn+r)\equiv 0\pmod{p},\quad \textrm{for all $n\in \mathbf N$}$$
implies that
$$c\phi_{pN+k}(pn+r)\equiv 0\pmod{p}, \quad \textrm{for all $n\in \mathbf N$}.$$
In \cite[(10.3)]{Andrews}, Andrews showed that for all integers $n\geq 0$, \begin{equation}\label{Andrews-cphi-2-cong} c\phi_2(5n+3)\equiv 0\pmod{5}.\end{equation}
Applying the result of Garvan and  Sellers with $p=5, r=3, N=1$ and $k=2$, we complete the proof of Theorem
\ref{c-phi-7-mod-5}.
\end{rem}

Our next set of congruences are consequences of \eqref{c-phi-7-product}.

\begin{theorem}\label{cphi7-cong-thm}
For any integer $n\ge 0$, we have
\begin{align}\label{cphi7-cong}
c\phi_{7}(7n+3)\equiv c\phi_{7}(7n+5) \equiv c\phi_{7}(7n+6) \equiv 0 \pmod{7^3}.
\end{align}
\end{theorem}

\begin{proof}
From \eqref{c-phi-7-product}, we find that
\begin{align}
\sum_{k=0}^{\infty}c\phi_{7}(k)q^k \equiv \frac{1}{(q^7;q^7)_\infty}+49q\frac{(q^7;q^7)_\infty^{3}}{(q;q)_\infty^{4}} \pmod{7^3}.
\end{align}
Let
\[q\frac{(q^7;q^7)_\infty^{3}}{(q;q)_\infty^{4}}=\sum_{j=0}^{\infty}a(j)q^j.\]
Then
\begin{align}\label{relate-an}
\sum_{n=0}^{\infty}c\phi_{7}(7n+r)q^{n}\equiv 49\sum_{n=0}^{\infty} a(7n+r)q^n \pmod{7^3}, \quad 1\le r \le 6.
\end{align}
By the binomial theorem, we find that
\begin{align}
\sum_{j=0}^{\infty}a(j)q^j\equiv q (q^7;q^7)_\infty^{2}(q;q)_\infty^{3} \equiv (q^7;q^7)_\infty^{2}\left(\sum_{i=0}^{\infty}(-1)^{i}(2i+1)q^{i(i+1)/2+1}\right) \pmod{7}.
\end{align}
Since
$$1+\frac{i(i+1)}{2}\equiv 0, 1, 2 \,\,\text{or}\,\, 4\pmod{7},$$
we deduce that
\begin{align}\label{an-cong}
a(7n+3)\equiv a(7n+5) \equiv a(7n+6) \equiv 0 \pmod{7}.
\end{align}
Combining \eqref{relate-an} with \eqref{an-cong} we complete the proof of \eqref{cphi7-cong}.
\end{proof}

\subsection{Case $k=11$}

\medskip

\begin{theorem}\label{cphi11-thm}
We have
\begin{equation}\label{c-phi-11-right}
\mathrm{C}\Phi_{11}(q)=\frac{1}{(q^{11};q^{11})_\infty}+11\sum_{j=1}^{\infty}p(11j-5)q^j.
\end{equation}
\end{theorem}

\begin{proof}
By Theorem \ref{generating-thm}, we know that
$\mathfrak{A}_{11}(q)\in M_{5}\left(\Gamma_{0}(11), \left(\frac{-11}{\cdot} \right) \right)$.
The dimension of $M_{5}\left(\Gamma_{0}(11), \left(\frac{-11}{\cdot} \right) \right)$ is 5 \cite[Theorem 1.34]{Ono} and this space is spanned by
the modular forms
\begin{align*}
&\frac{\eta^{11}(11\tau)}{\eta(\tau)}, \quad \frac{\eta^{11}(\tau)}{\eta(11\tau)},\quad
(q;q)_\infty^{11} \sum_{j=1}^\infty p(11j-5) q^{j},\\
&\sum_{m=1}^\infty\sum_{d=1}^\infty \left(\frac{d}{11}\right) m^4 q^{md}\quad\text{and}\quad
\frac{1275}{11}+\sum_{m=1}^\infty\sum_{d=1}^\infty \left(\frac{d}{11}\right)d^4 q^{md}.
\end{align*}
By comparing the coefficients of $\mathfrak{A}_{11}(q)$ with those of the five modular forms above, we deduce that
\[\mathfrak{A}_{11}(q)=\frac{\eta^{11}(\tau)}{\eta(11\tau)}+11 (q;q)_\infty^{11} \sum_{j=1}^\infty p(11j-5) q^{j}.\]
This proves \eqref{c-phi-11-right}.
\end{proof}

It is immediate that \eqref{c-phi-11-right} implies \eqref{Analogue11}.
There is no analogue of
\eqref{Phi-5-E} and \eqref{Phi-7-E} for $k=11$ but an analogue for  \eqref{c-phi-5} and
\eqref{c-phi-7-theta} exists. This expression is complicated and we will give such identities if we do not have other
representations for $(q^k;q^k)_\infty\mathrm{C}\Phi_k(q)$ when $k$ is composite (see Section 6).

\subsection{Case $k=13$}
\begin{theorem}\label{cphi13-thm}
We have
\begin{align}
\mathrm{C}\Phi_{13}(q)&=\frac{1}{(q^{13};q^{13})_\infty}+169\Big(q\frac{(q^{13};q^{13})_\infty}{(q;q)_\infty^{2}}
+36q^2\frac{(q^{13};q^{13})_\infty^{3}}{(q;q)_\infty^{4}}+494q^3\frac{(q^{13};q^{13})_\infty^{5}}{(q;q)_\infty^{6}} \notag \\ & \qquad\qquad\qquad+3380q^4\frac{(q^{13};q^{13})_\infty^{7}}{(q;q)_\infty^{8}} +13182q^5\frac{(q^{13};q^{13})_\infty^{9}}{(q;q)_\infty^{10}}\notag \\
&\qquad\qquad\qquad +28561q^6\frac{(q^{13};q^{13})_\infty^{11}}{(q;q)_\infty^{12}}
+28561q^7\frac{(q^{13};q^{13})_\infty^{13}}{(q;q)_\infty^{14}} \Big)\label{cphi13-product}
\\
=&\frac{1}{(q^{13};q^{13})_\infty}+13\sum_{j=1}^\infty p(13j-7)q^{j}+26q\frac{(q^{13};q^{13})_\infty}{(q;q)_\infty^{2}}.\label{c-phi-13-right}\end{align}
\end{theorem}

\begin{proof}
From the discussion at the end of Section \ref{sub5},
we know that
 $$(q^{13};q^{13})_\infty\mathrm{C}\Phi_{13}(q)$$
is a modular function invariant under $\Gamma_0(13)$ and
since modular functions invariant under $\Gamma_0(13)$ are rational functions of
$H = \eta^2(13\tau)/\eta^2(\tau)$ \cite{Kondo}, we deduce that
\begin{align*}
&(q^{13};q^{13})_\infty\mathrm{C}\Phi_{13}(q)\\
&=
1+169\bigg(H+36H^2+494H^3+3380H^4
+13182 H^5+28561 H^6+28561 H^7\bigg)
\end{align*}
and \eqref{cphi13-product} follows.

Around 1939, motivated by Ramanujan's identities \eqref{Ramapart5} and \eqref{Ramapart7},
 H. Zuckerman \cite[Eq. (1.15)]{Zuckerman} discovered that
\begin{align}\label{Zuckerman-id}
\sum_{j=1}^\infty p(13j-7)q^j=11q&\frac{(q^{13};q^{13})_\infty}{(q;q)_\infty^{2}}+468q^2\frac{(q^{13};q^{13})_\infty^{3}}{(q;q)_\infty^{4}}
+6422q^3\frac{(q^{13};q^{13})_\infty^{5}}{(q;q)_\infty^{6}}\nonumber \\ &+43940q^4\frac{(q^{13};q^{13})_\infty^7}{(q;q)_\infty^{8}}
+171366q^5\frac{(q^{13};q^{13})_\infty^{9}}{(q;q)_\infty^{10}}\nonumber\\
&+371293q^6\frac{(q^{13};q^{13})_\infty^{11}}{(q;q)_\infty^{12}}+371293q^7\frac{(q^{13};q^{13})_\infty^{13}}{(q;q)_\infty^{14}}.
\end{align}
Using \eqref{Zuckerman-id} to simplify \eqref{cphi13-product}, we deduce
that \[{\rm C}\Phi_{13}(q) = \frac{1}{(q^{13};q^{13})_\infty}+13\sum_{j=1}^\infty p(13j-7)q^{j}+26q\frac{(q^{13};q^{13})_\infty}{(q;q)_\infty^{2}}\]
and this yields \eqref{c-phi-13-right}.
\end{proof}

Identity \eqref{c-phi-13-right} immediately implies
\eqref{Analogue13}.

We observe that the appearance of
$$\sum_{j=1}^\infty p(13j-7)q^{j}$$ simplifies \eqref{cphi13-product}, leading to \eqref{c-phi-13-right} with only three terms on the right hand side. Identity \eqref{c-phi-13-right} is clearly an analogue of Kolitsch's identities
 \eqref{c-phi-5-right} and \eqref{c-phi-7-right}.

In Section 4, we will prove identities involving both $\mathrm{C}\Phi_k(q)$ and
$$\sum_{j=1}^\infty p\left(kj-\frac{k^2-1}{24}\right)q^j$$ when $k>3$ is a prime.
This method appears to yield the simplest (in terms of the number of modular forms
involved)
representation of $\mathrm{C}\Phi_k(q)$ for any prime $k>3$ and it does not involve the construction of basis for
$$M_{(k-1)/2}\Big(\Gamma_{0}(k), \big(\frac{(-1)^{(k-1)/2}k}{\cdot} \big) \Big).$$
Constructing such basis could get complicated for large $k$, as we shall see in the next subsection.

\subsection{Case $k=17$}
\quad
\medskip

Let
\begin{equation}\label{Ea}
\mathcal{E}_{a}(\tau)=q^{17B_{2}(a/17)/2}\prod\limits_{m=1}^{\infty}(1-q^{17(m-1)+a})(1-q^{17m-a}),
\end{equation} where $B_{2}(x)=x^2-x+1/6$.
Let
\begin{align*}
&f_{17,1}=\eta^{7}(\tau)\eta(17\tau), \quad f_{17,2}(\tau)=\eta(\tau)\eta^{7}(17\tau), \\
&g_{17,1}(\tau)=\frac{1}{8}(17E_{2}(17\tau)-E_{2}(\tau)), \\
&g_{17,2}(\tau)=\frac{1}{4}\eta^{4}(17\tau)\sum_{k=0}^{7}\mathcal{E}_{2\cdot 3^{k}}(\tau)\mathcal{E}_{14\cdot 3^{k}}(\tau)\mathcal{E}_{4\cdot 3^{k}}(\tau)^{2}\mathcal{E}_{12\cdot 3^{k}}(\tau)\mathcal{E}_{6\cdot 3^{k}}(\tau)\mathcal{E}_{10\cdot 3^{k}}(\tau)^2\mathcal{E}_{8\cdot 3^{k}}(\tau), \\
&h_{17,1}(\tau)=g_{17,1}^{2}(\tau), \quad h_{17,2}(\tau)=g_{17,1}(\tau)g_{17,2}(\tau), \quad h_{17,3}(\tau)=g_{17,2}^{2}(\tau),\\
&h_{17,4}(\tau)=\eta^{4}(\tau)\eta^{4}(17\tau), \quad h_{17,5}(\tau)=\frac{1}{24}\big( 289E_{4}(17\tau)-E_{4}(\tau)\big).
\end{align*}

From Theorem \ref{generating-thm}, we know $\mathfrak{A}_{17}(q) \in M_{8}\big(\Gamma_{0}(17), \big(\frac{17}{\cdot}\big)\big)$. By \cite[Theorem 1.34]{Ono}, we find that
$$\dim M_{8}\big(\Gamma_{0}(17), \big(\frac{17}{\cdot}\big)\big)=12.$$

Let
\begin{align*}
&B_{17,1}=f_{17,1}h_{17,1},\,\, B_{17,2}=f_{17,1}h_{17,2}, \,\, B_{17,3}=f_{17,1}h_{17,3}, \,\, B_{17,4}=f_{17,1}h_{17,4},\\
& B_{17,5}=f_{17,1}h_{17,5}, \,\, B_{17,6}=f_{17,2}h_{17,1}, \,\, B_{17,7}=f_{17,2}h_{17,2},  \,\, B_{17,8}=f_{17,2}h_{17,3},\\
&B_{17,9}=f_{17,2}h_{17,4}, \,\, B_{17,10}=f_{17,2}h_{17,5},  \,\, B_{17,11}=\frac{\eta^{17}(\tau)}{\eta(17\tau)},  \quad\text{and}\quad B_{17,12}=\frac{\eta^{17}(17\tau)}{\eta(\tau)}.
\end{align*}
One can verify that $\{B_{17,j}|1\le j \le 12\}$ forms a basis of $M_{8}\big(\Gamma_{0}(17), \big(\frac{17}{\cdot}\big)\big)$. By comparing the Fourier coefficients of $\mathfrak{A}_{17}(q)$ and $B_{17,j}, 1 \le j \le 12$, we deduce the following identity:
\begin{theorem}\label{cphi-17-thm}
We have
\begin{align}\label{CPhi-17}
\mathrm{C}\Phi_{17}(q)=&\frac{1}{(q;q)_{\infty}^{17}}\Big(\frac{1491529}{118}B_{17,1} -\frac{20931981}{236}B_{17,2} -\frac{117030839}{236}B_{17,3}\\
&+ \frac{78308596}{59}B_{17,4} -\frac{988669}{236}B_{17,5}+\frac{424841849}{59}B_{17,6} -\frac{10654955751}{236}B_{17,7} \nonumber \\
&-\frac{17109438979}{236}B_{17,8}+\frac{7515406274}{59}B_{17,9}+\frac{91750275}{236}B_{17,10}\nonumber \\
&+B_{17,11}+6975757441B_{17,12} \Big). \nonumber
\end{align}
\end{theorem}
Note that all the coefficients of $B_{17,j}$, $j\neq 11$, are divisible by $17^2$. Therefore,
\begin{align*}
\mathrm{C}\Phi_{17}(q)\equiv \frac{1}{(q^{17};q^{17})_{\infty}} \pmod{17^2},
\end{align*}
or equivalently,
\[c\phi_{17}(n) \equiv p(n/17) \pmod{17^2}.\]
This is a special case of Andrews' congruence \cite[Theorem 10.2 and Corollary 10.2]{Andrews}
\begin{align}\label{A-eq}
c\phi_{p}(n)\equiv c\phi_{1}(n/p) \pmod{p^2},
\end{align}
which is true for all primes $p$.

In the next section, we will provide an analogue for Kolitsch's identities
 \eqref{c-phi-5-right} and \eqref{c-phi-7-right} for $\mathrm{C}\Phi_{17}(q)$.

\section{$k$-colored generalized Frobenius partitions and ordinary partitions}

Kolitsch's identities \eqref{Kolitsch5}, \eqref{Kolitsch7}, and
Andrews' congruence \eqref{A-eq} show a close relation
between $k$-colored generalized Frobenius partitions and ordinary
partitions. In this section, we will give a more precise description
of the relation and prove
 \eqref{c-phi-5-right}, \eqref{c-phi-7-right}, \eqref{c-phi-11-right}  and \eqref{c-phi-13-right} in a uniform way.
 We will also give an alternative representation for $\mathrm{C}\Phi_{17}(q)$ and illustrate for any prime
 $\ell>3$,
 a general procedure to express $\mathrm{C}\Phi_\ell(q)$ in terms of other modular functions, one of
 which involves generating functions for $p(\ell n -(\ell^2-1)/24)$.

Let
\begin{align*}
F(\tau)\Big|\begin{pmatrix} a& b\\ c&d \end{pmatrix} := F\left(\frac{a\tau+b}{c\tau+d}\right).
\end{align*}
Let $\ell$ be a prime $\ge 5$ and let $\mathcal A_\ell(\tau)$ denote the function
$\mathfrak A_\ell(q)$ when the function $\mathfrak A_\ell(q)$ is viewed as a function of $\tau$ with $q=e^{2\pi i\tau}$. By Theorem \ref{generating-thm},
$$
  \mathfrak A_\ell(q)=\mathcal A_\ell(\tau)=(q;q)_\infty^\ell\mathrm{C}\Phi_\ell(q)
 =\sum_{m_1,\ldots,m_{\ell-1}\in\mathbf Z}q^{Q(m_1,\ldots,m_{\ell-1})}
$$
is a modular form of weight $(\ell-1)/2$ with character
$\chi_{(-1)^{(\ell-1)/2}\ell}$ on $\Gamma_0(\ell)$, where
$Q(m_1,\ldots,m_{\ell-1})$ is the quadratic form defined by \eqref{Q}
and $\chi_d$ is
the character defined by $\chi_d(\cdot)=\left(\frac d\cdot\right)$.
It follows that
$$
  f_\ell(\tau)=\frac{\eta(\ell \tau)}{\eta(\tau)^\ell}\mathcal A_\ell(\tau)
$$
is a modular function on $\Gamma_0(\ell)$. On the other hand,
$\eta(\ell^2\tau)/\eta(\tau)$ is a modular function on $\Gamma_0(\ell^2)$
and by a lemma of A.O.L. Atkin and J. Lehner \cite[Lemma 7]{Atkin-Lehner}, we find that
$$
  \frac{\eta(\ell^2\tau)}{\eta(\tau)}\Big|U_\ell:=
  \frac1\ell\sum_{k=0}^{\ell-1}\frac{\eta(\ell^2\tau)}{\eta(\tau)}\Big|
  \begin{pmatrix}1&k\\0&\ell\end{pmatrix}
 =(q^\ell;q^\ell)_\infty
  \sum_{j=1}^\infty p\left(\ell j-\frac{\ell^2-1}{24}\right)q^j
$$
is also a modular function on $\Gamma_0(\ell)$.

Set%
\begin{equation}\label{g-ell}
  g_\ell(\tau)=1+\ell \frac{\eta(\ell^2\tau)}{\eta(\tau)}\Big|U_\ell=1+\ell(q^\ell;q^\ell)_\infty
  \sum_{j=1}^\infty p\left(\ell j-\frac{\ell^2-1}{24}\right)q^j.
\end{equation}

We now compare the analytic behaviors of $f_\ell(\tau)$ and $g_\ell(\tau)$ at cusps associated with $\Gamma_0(\ell)$.

\begin{lemma} \label{f-g} Let   $\ell\ge 5$ be a prime and
$$\delta_\ell =\frac{\ell^2-1}{24}.$$
At the cusp $\infty$, we have
$$
  f_\ell(\tau)-g_\ell (\tau)=\begin{cases}
  \ell(\ell-p(\ell-\delta_\ell))q+O(q^2), &\text{if }\ell\le 23, \\
  \displaystyle\ell^2q+\left(\frac14\ell^2(\ell^2-2\ell+9)-
  \ell p(2\ell-\delta_\ell)\right)q^2+O(q^3),
 &\text{if }29\le\ell\le 47, \\
  \displaystyle\ell^2q+\frac14\ell^2(\ell^2-2\ell+9)q^2+O(q^3),
 &\text{if }\ell\ge 53. \end{cases}
$$
At the cusp $0$, we have
$$
  (f_\ell(\tau)-g_\ell(\tau))\Big|\begin{pmatrix}0&-1\\\ell&0\end{pmatrix}
 =\begin{cases}
  O(q), &\text{if }\ell=5,7,11,\\
  2q^{-1}-4+O(q), &\text{if }\ell=13,\\
  2q^{-(\ell^2-1)/24+(\ell-1)/2}(1-q+O(q^2)), &\text{if }\ell\ge 17.
  \end{cases}
$$
\end{lemma}

\begin{proof}
  It is clear from the definition of  $g_{\ell}(\tau)$ that
  \begin{equation} \label{g at infinity}
  \begin{split}
    g_{\ell}(\tau)&=1+{ \ell (q^{\ell};q^{\ell})_{\infty}}\sum_{n\ge(\ell^2-1)/24\ell} p(\ell n-\delta_\ell)q^n \\
  &=\begin{cases}
    1+\ell p(\ell-\delta_\ell)q+O(q^2), &\text{if }\ell\le 23, \\
    1+\ell p(2\ell-\delta_\ell)q^2+O(q^3),
    &\text{if }29\le\ell\le 47, \\
    1+O(q^3), &\text{if }\ell\ge 53. \end{cases}
  \end{split}
  \end{equation}
  On the other hand, we have
  \begin{equation*}
  \begin{split}
    Q(m_1,\ldots,m_{\ell-1})&=(m_1^2+\cdots+m_{\ell-1}^2)
    +\frac12\sum_{i\neq j}m_im_j \\
  &=\frac12(m_1^2+\ldots+m_{\ell-1}^2)+\frac12(m_1+\cdots+m_{\ell-1})^2.
  \end{split}
  \end{equation*}
   From this, we see that $Q(m_1,\ldots,m_{\ell-1})=1$ if and only if
  exactly one of $m_j$ is $\pm 1$ and the other are all $0$, or
  $m_i=1$ and $m_j=-1$ for some $i,j$ with $i\neq j$ and all others
  are $0$. Likewise, we can check that $Q(m_1,\ldots,m_{\ell-1})=2$ if
  and only if there are two $1$'s and two $-1$'s among $m_j$, or
  there are two $1$'s and one $-1$ among $m_j$, or there are two
  $-1$'s and one $1$ among $m_j$. Thus, the number of integer
  solutions of $Q(m_1,\ldots,m_{\ell-1})=2$ is
  $$
    \frac14(\ell-1)(\ell-2)(\ell-3)(\ell-4)+2\cdot
    \frac12(\ell-1)(\ell-2)(\ell-3)=\frac14\ell(\ell-1)(\ell-2)(\ell-3).
  $$
  Consequently, we have
  $$
    \mathcal A_\ell(\tau)=1+\ell(\ell-1)q+\frac14\ell(\ell-1)
    (\ell-2)(\ell-3)q^2+\cdots
  $$
  and
  \begin{equation*}
  \begin{split}
    f_\ell(\tau)&=\frac{(q^\ell;q^\ell)_\infty}{(1-q-q^2+\cdots)^\ell}
   \left(1+\ell(\ell-1)q+\frac14\ell(\ell-1)(\ell-2)(\ell-3)q^2+\cdots\right)\\
  &=1+\ell^2q+\frac14\ell^2(\ell^2-2\ell+9)q^2+\cdots.
  \end{split}
  \end{equation*}
  Together with \eqref{g at infinity}, this yields the first half of
  the lemma.
  We next consider the analytic behavior of $f_\ell(\tau)-g_\ell(\tau)$ at $0$.

  Recall that if $\Lambda$ is an even integral lattice of rank $n$ and
  $\Lambda'$ is its dual lattice, then their theta series
  $\theta_\Lambda(\tau)$ and $\theta_{\Lambda'}(\tau)$ are related by the
  transformation formula (see \cite[Proposition 16, Chapter VII]{Serre})
  \begin{equation} \label{theta transformation}
    \theta_{\Lambda'}(-1/\tau)=\left(\frac \tau i\right)^{n/2}
    \nu(\Lambda)\theta_\Lambda(\tau),
  \end{equation}
  where $\nu(\Lambda)$ is the volume of the lattice $\Lambda$. Here, we let
  $\Lambda$ be the lattice whose Gram matrix is $\ell
  A_{\ell-1}^{-1}$, where $A_n$ and $A_n^{-1}$ are given by
  \eqref{An} and \eqref{An-1}, respectively.
  The determinant of $\ell A_{\ell-1}^{-1}$ is
  $\ell^{\ell-1}/\det(A_{\ell-1})=\ell^{\ell-2}$. Hence
  $$
    \nu(\Lambda)=\ell^{\ell/2-1}.
  $$
  Let $\mathcal B_\ell(\tau)$ be the theta series of $\Lambda$ and observe
  that the theta series of $\Lambda'$ is  $\mathcal A_{\ell}(\tau/\ell)$.
 Thus, by \eqref{theta transformation}, we have
  $$
    {\mathcal A_{\ell}}\left(-\frac1{\ell \tau}\right)
   =\ell^{\ell/2-1}\left(\frac \tau i\right)^{(\ell-1)/2}
    \mathcal B_\ell(\tau).
  $$
  Together with
  $$
    \eta\left(-\frac1{\ell \tau}\right)
   =\sqrt{\frac{\ell \tau}i}\eta(\ell \tau), \quad\text{and}\quad
  \eta\left(-\frac1 \tau\right)=\sqrt{\frac \tau i}\eta(\tau),
  $$
we deduce that
  \begin{equation} \label{equation: f at 0}
    f_\ell(\tau)\Big|\begin{pmatrix}0&-1\\\ell&0\end{pmatrix}
   =\frac1\ell\frac{\eta(\tau)}{\eta(\ell \tau)^\ell}
    \mathcal B_\ell(\tau).
  \end{equation}

  We now consider $g_\ell(-1/\ell \tau)$.

  We have
  $$
    \ell\left(\frac{\eta(\ell^2\tau)}{\eta(\tau)}\Big|U_\ell\right)
    \Big|\begin{pmatrix}0&{-1}\\ \ell&0\end{pmatrix}
   =\sum_{k=0}^{\ell-1}\frac{\eta(\ell^2\tau)}{\eta(\tau)}\Big|
    \begin{pmatrix}{k\ell}&{-1}\\{\ell^2}&0\end{pmatrix}.
  $$
  For $k=0$, the transformation formula for $\eta(\tau)$ yields
  \begin{equation} \label{equation: g at 0 1}
    \frac{\eta(\ell^2\tau)}{\eta(\tau)}\Big|
    \begin{pmatrix}0&{-1}\\{\ell^2}&0\end{pmatrix}
   =\frac1\ell\frac{\eta(\tau)}{\eta(\ell^2\tau)}.
  \end{equation}
  For $1 \le k \le \ell-1$, we find that
  \begin{align}\label{addeq-eta-1}
    \eta\left(\ell^2\frac{k\ell \tau-1}{\ell^2\tau}\right)
   =\eta\left(k\ell-\frac1 \tau \right)
   =e^{2\pi ik\ell/24}\eta(-1/\tau)=e^{2\pi ik\ell/24}\sqrt{\frac \tau i}\eta(\tau).
  \end{align}

  Next, since $(k,\ell)=1$, there exist integers $a$ and $k'$ such that $kk'-a\ell=1$.
  This implies that
  \begin{equation}\label{addeq-eta-2}
    \eta\left(\frac{k\ell \tau-1}{\ell^2\tau}\right)
  =\eta\left(\tau-\frac{k'}{\ell}\right)\Big|
    \begin{pmatrix} k & a\\{\ell}&k'\end{pmatrix}\\
    =\left(\frac{k'}\ell\right) i^{(1-\ell)/2}
     e^{2\pi i\ell(k+k')/24}\sqrt{\frac{\ell \tau}i}
     \eta\left(\tau-\frac{k'}\ell\right).
  \end{equation}
  It follows {from \eqref{addeq-eta-1} and \eqref{addeq-eta-2}} that
  \begin{equation*}
  \begin{split}
    \frac{\eta(\ell^2\tau)}{\eta(\tau)}\Big|
    \begin{pmatrix}{k\ell}&{-1}\\{\ell^2}&0\end{pmatrix}
  &=\frac1{\sqrt\ell}\left(\frac{k'}\ell\right)
    i^{(\ell-1)/2}e^{-2\pi i\ell k'/24}\frac{\eta(\tau)}
    {\eta(\tau-k'/\ell)} \\
  &=\frac1{\sqrt\ell}\left(\frac{k'}\ell\right)
    i^{(\ell-1)/2}e^{-2\pi imk'/\ell}+O(q),
\end{split}
\end{equation*}
where $m=(\ell^2-1)/24$. Hence,
\begin{equation} \label{equation: g at 0 2}
\begin{split}
  \sum_{k=1}^{\ell-1}\frac{\eta(\ell^2\tau)}{\eta(\tau)}\Big|
  \begin{pmatrix}{k\ell}&{-1}\\{\ell^2}&0\end{pmatrix}
 &=\frac{i^{(\ell-1)/2}}{\sqrt\ell}\sum_{k=1}^{\ell-1}
   \left(\frac{k'}\ell\right)
   e^{-2\pi imk'/\ell}+O(q) \\
 &=\frac{i^{(\ell-1)/2}}{\sqrt\ell}\left(\frac{-m}\ell\right)
   \sum_{n=1}^{\ell-1}
   \left(\frac n\ell\right)e^{2\pi in/\ell}+O(q) \\
 &=i^{(\ell-1)/2}\left(\frac{-m}\ell\right)\begin{cases}
   1+O(q), &\text{if }\ell\equiv 1\mod 4, \\
   i+O(q), &\text{if }\ell\equiv 3\mod 4, \end{cases} \\
 &=\left(\frac8\ell\right)\left(\frac{-m}\ell\right)+O(q) \\
 &=\left(\frac{12}\ell\right)+O(q),
\end{split}
\end{equation}
where we have used Gauss' result \cite[Section 9.10]{Apostol} in our third equality.
Combining \eqref{equation: f at 0}, \eqref{equation: g at 0 1}, and
\eqref{equation: g at 0 2}, we find that
\begin{equation} \label{equation: f-g at 0}
\begin{split}
  (f_\ell(\tau)-g_\ell(\tau))\Big|\begin{pmatrix}0&{-1}\\ \ell&0\end{pmatrix}
&=\frac1\ell q^{-(\ell^2-1)/24}(q;q)_\infty\left(
  \frac{{ \mathcal B_{\ell}(\tau)}}{(q^\ell;q^\ell)_\infty^\ell}
 -\frac1{(q^{\ell^2};q^{\ell^2})_\infty}
  \right)\\
&\qquad-\left(\frac{12}\ell\right)-1+O(q).
\end{split}
\end{equation}
We now claim that
\begin{equation} \label{B expansion}
  \mathcal B_\ell(\tau)=1+2\ell q^{(\ell-1)/2}+\cdots,
\end{equation}
so that
$$
  \frac{\mathcal B_\ell(\tau)}{(q^\ell;q^\ell)_\infty^\ell}
 -\frac1{(q^{\ell^2};q^{\ell^2})_\infty}=2\ell q^{(\ell-1)/2}+\cdots.
$$

Recall that $\mathcal B_\ell(\tau)$ is defined to be the theta series
associated to the lattice whose Gram matrix is $\ell A_{\ell-1}^{-1}$,
where $A_n^{-1}$ is given by \eqref{An-1}. In other words, we have
$$
  \mathcal B_\ell(\tau)=\sum_{m_1,\ldots,m_{\ell-1}\in\mathbf Z}
  q^{Q'(m_1,\ldots,m_{\ell-1})}, \quad q=e^{2\pi i\tau},
$$
where
\begin{equation*}
\begin{split}
  Q'(m_1,\ldots,m_{\ell-1})
&=\frac{\ell-1}2(m_1^2+\cdots+m_{\ell-1}^2)-{  \frac{1}{2}}\sum_{i\neq j}m_im_j \\
&=\frac12\left(\ell(m_1^2+\cdots+m_{\ell-1}^2)-(m_1+\cdots+m_{\ell-1})^2\right).
\end{split}
\end{equation*}
For each $(m_1,\ldots,m_{\ell-1})\in\mathbf
Z^{\ell-1}\backslash\{0\}$, let $r$ be the number of nonzero entries
in the tuple. By the Cauchy-Schwarz inequality, we have
$$
  (m_1+\cdots+m_{\ell-1})^2\le r(m_1^2+\cdots+m_{\ell-1}^2).
$$
Then
$$
  Q'(m_1,\ldots,m_{\ell-1})\ge\frac12
  (\ell-r)(m_1^2+\cdots+m_{\ell-1}^2)\ge\frac12(\ell-r)r
  \ge\frac{\ell-1}2.
$$
Therefore, the coefficient of $q^j$ in $\mathcal B_\ell(\tau)$ vanishes for
$j=1,\ldots,(\ell-1)/2-1$. Also, the contribution to the
$q^{(\ell-1)/2}$ term comes from the cases where $r=1$ or $r=\ell-1$
and equality holds for each of the inequality above.
In other words, the contribution to $q^{(\ell-1)/2}$ comes from the
tuples where exactly one of $m_j$ is $\pm 1$ and all the
others are $0$ or $(m_1,\ldots,m_{\ell-1})=\pm(1,\ldots,1)$. We
conclude that the coefficient of $q^{(\ell-1)/2}$ in $\mathcal B_\ell(\tau)$
is $2\ell$. This proves the claim \eqref{B expansion}.

For the cases $\ell=5$ and $\ell=7$, we have
$(\ell^2-1)/24<(\ell-1)/2$ and $\left(\frac{12}\ell\right)=-1$.
Therefore,
\begin{equation} \label{equation: temp}
  \left(f_\ell(\tau)-g_\ell(\tau)\right)\Big|\begin{pmatrix}0&-1\\
  \ell&0\end{pmatrix}=O(q).
\end{equation}
When $\ell=11$, we have $(\ell^2-1)/24=(\ell-1)/2$ and
$\left(\frac{12}\ell\right)=1$. Again, \eqref{equation: f-g at 0}
implies that \eqref{equation: temp} holds in this case. For other
cases, we note that in general, we have
$$
  \mathcal B_\ell(\tau)=1+2\ell q^{(\ell-1)/2}+\ell(\ell-1)q^{\ell-2}+\cdots,
$$
and hence,
$$
q^{-(\ell^2-1)/24}(q;q)_\infty\left(
  \frac{\mathcal B_\ell(\tau)}{(q^\ell;q^\ell)_\infty^\ell}
 -\frac1{(q^{\ell^2};q^{\ell^2})_\infty}
  \right)=2\ell q^{-(\ell^2-1)/24+(\ell-1)/2}(1-q-q^2+\cdots)
$$
for $\ell\ge 11$. When $\ell=13$, we have
$-(\ell^2-1)/24+(\ell-1)/2=-1$ and $\left(\frac{12}{13}\right)=1$.
Then from \eqref{equation: f-g at 0}, we deduce that
$$
  \left(f_\ell(\tau)-g_\ell(\tau)\right)\Big|\begin{pmatrix}0&-1\\
  \ell&0\end{pmatrix}=2q^{-1}-4+O(q).
$$
For other primes $\ell\ge17$, \eqref{equation: f-g at 0} yields
$$
  \left(f_\ell(\tau)-g_\ell(\tau)\right)\Big|\begin{pmatrix}0&-1\\
  \ell&0\end{pmatrix}=2q^{-(\ell^2-1)/24+(\ell-1)/2}(1-q+O(q^2))
$$
instead. This completes the proof of the lemma.
\end{proof}

\begin{theorem}\label{S5main} Let $\ell\ge 5$ be a prime. Let
  $$
    f_\ell(\tau)=(q^\ell;q^\ell)_\infty\mathrm{C}\Phi_\ell(q),
  $$
  and
  $$
    g_\ell(\tau)=1+\ell(q^\ell;q^\ell)_\infty\sum_{n=1}^\infty
    p\left(\ell n-\frac{\ell^2-1}{24}\right)q^n.
  $$
  % $\delta_\ell$ is the unique integer
  %between $0$ and $\ell$ such that $24\delta_\ell\equiv-1\mod\ell$.
  \begin{enumerate}[{\rm (a)}]
  \item \label{a} If $\ell=5,7,11$, then $f_\ell(\tau)=g_\ell(\tau)$.
  \item \label{b} If $\ell=13$, then
  $$
    f_{13}(\tau)=g_{13}(\tau)+26\frac{\eta^2(13\tau)}{\eta^2(\tau)}.
  $$
  \item \label{c}  If $\ell\ge 17$, then
\begin{equation}\label{h-ell}
    h_\ell(\tau)=f_\ell(\tau)-g_\ell(\tau)-2\ell^{(\ell-11)/2}\left(
    \frac{\eta(\ell \tau)}{\eta(\tau)}\right)^{\ell-11}
  \end{equation}
  is a modular function on $\Gamma_0(\ell)$ with a zero at
  $\infty$ and a pole of order $(\ell+1)(\ell-13)/24$ at $0$ and
  $$
    h_\ell(\tau)(\eta(\tau)\eta(\ell \tau))^{\ell-13}
  $$
  is a holomorphic modular form of weight $\ell-13$ with a zero of
  order $(\ell-1)(\ell-11)/24$ at $\infty$.
% $$s\ge \frac{(\ell-1)(\ell-11)}{24}$$ at $\infty$.
  \item \label{d} We have
  $$
    h_\ell(\tau)\equiv 0\mod\begin{cases}
    170, &\text{when }\ell=17, \\
    266, &\text{when }\ell=19, \\
    506, &\text{when }\ell=23. \end{cases}
  $$

  \item \label{e} For any prime $\ell>11$,
  $$h_\ell(\tau)\equiv \ell F_\ell(\tau) \pmod{\ell^2}$$ where
  $F_\ell(\tau)$ is a non-zero modular form of weight $\ell-1$ on $\text{SL}(2,\mathbf Z)$.

  \end{enumerate}
\end{theorem}

\begin{proof} We first remark that the functions $f_\ell(\tau)$ and $g_\ell(\tau)$ are both
  holomorphic on the upper half-plane. Thus, to prove that $f_\ell(\tau)=g_\ell(\tau)$
  for the cases $\ell=5,7,11$, we only need to verify that $f_\ell(\tau)-g_\ell(\tau)$
  does not have poles at cusps and $f_\ell(\tau)-g_\ell(\tau)$ vanishes at one
  particular point in these three cases.
  Indeed, by Lemma \ref{f-g}, $f_\ell(\tau)-g_\ell(\tau)$ vanishes at both cusps in the three cases
 since  $p(\ell-\delta_\ell)=\ell$ for $\ell=5, 7 $ and $11$.
 This proves \eqref{a}.
 We remark that in fact, it suffices to know that $f_\ell(\tau)-g_\ell(\tau)$ has no pole at
 the cusp 0 since it would mean that $f_\ell(\tau)-g_\ell(\tau)$ is a constant. Since the expansion at $\infty$ begins
 with $\ell\left(\ell-p(\ell-\delta_\ell)\right)q$, the only possibility that $f_\ell(\tau)-g_\ell(\tau)$ is  a constant is when
 $p(\ell-\delta_\ell)=\ell.$ In other words, without listing out the partitions of $5, 7$ and 11, we know that
 $p(4)=5, p(5)=7$ and $p(6)=11$.

  We next consider the case $\ell=13$. By Lemma \ref{f-g},
  the Fourier expansion of $f_\ell(\tau)-g_\ell(\tau)$ at $0$ is
  $$
    2q^{-1}-4+\cdots.
  $$
  Now we observe that $\eta(13\tau)^2/\eta(\tau)^2$ is also a modular
  function on $\Gamma_0(13)$ and satisfies
  $$
    \frac{\eta^2(13\tau)}{\eta^2(\tau)}\Big|\begin{pmatrix}0&-1\\13&0
    \end{pmatrix}=\frac1{13}\frac{\eta^2(\tau)}{\eta^2(13\tau)}
   =\frac1{13}(q^{-1}-2+\cdots).
  $$
  Therefore, $f_\ell(\tau)-g_\ell(\tau)-26\eta(13\tau)^2/\eta(\tau)^2$ is a modular function
  on $\Gamma_0(13)$ that has no poles and vanishes at the cusps.
  We conclude that $f_\ell(\tau)-g_\ell(\tau)-26\eta(13\tau)^2/\eta(\tau)^2$ is identically
  $0$ and the proof of \eqref{b} is complete.

  Similarly, for primes $\ell\ge 17$, using Lemma \ref{f-g}
  and the transformation formula of $\eta(\tau)$, we find that
  \begin{equation*}
  \begin{split}
    h_\ell(\tau)\Big|\begin{pmatrix}0&-1\\\ell&0\end{pmatrix}
  &=2q^{-(\ell^2-1)/24+(\ell-1)/2}((1-q+O(q^2))-(1-(\ell-11)q+O(q^2)))\\
  &=2(\ell-12)q^{-(\ell+1)(\ell-13)/24}+\cdots.
  \end{split}
  \end{equation*}
  Therefore, $h_\ell(\tau)$ has a pole of order $(\ell+1)(\ell-13)/24$ for
  $\ell\ge 17$. From Lemma \ref{f-g}, it is clear that $h_\ell(\tau)$
  has a zero at $\infty$. It follows that
  $h_\ell(\tau)(\eta(\tau)\eta(\ell \tau))^{\ell-13}$ is a holomorphic
  modular form of weight $\ell-13$ on $\Gamma_0(\ell)$ and this
  completes the proof of \eqref{c}.

  The congruences in \eqref{d} can be verified using Sturm's criterion
  \cite{Sturm}.

Next, observe from \eqref{andrews-cphip} that
$$f_\ell(\tau) \equiv 1\pmod{\ell^2}.$$ For $\ell>13$,
$$h_\ell(\tau) \equiv f_\ell(\tau)-g_\ell(\tau) \equiv -\ell(q^\ell;q^\ell)_\infty \sum_{j=1}^\infty p\left(\ell j-\frac{\ell^2-1}{24}\right)q^j \pmod{\ell^2}.$$
It is known that
(see \cite[p. \!\!\! 157, Corollary 5.15.1]{Andrews-Berndt-III} for a proof given by J.P. Serre)
$$(q^\ell;q^\ell)_\infty \sum_{j=1}^\infty p\left(\ell j-\frac{\ell^2-1}{24}\right)q^j = F_\ell(\tau) +\ell E_\ell(\tau)$$ where
$F_\ell(\tau)$ is a cusp form on $\text{SL}(2,\mathbf Z)$ of weight $\ell-1$.
This implies that
$$h_\ell(\tau) \equiv -\ell F_\ell(\tau) \pmod{\ell^2}.$$
The fact that $F_\ell(\tau)$ is non-zero follows from the result of S. Ahlgren and M. Boylan \cite[Theorem 1]{Ahlgren-Boylan}.
\end{proof}

We now give another representation for $\mathrm{C}\Phi_{17}(q)$.
Let
$$h_1(\tau) = \eta^8(17\tau)\sum_{k=0}^7\mathcal E_{3^k}^{-1}(\tau)\mathcal E_{2\cdot 3^k}^{-2}(\tau)\mathcal E^{-1}_{5\cdot 3^k}(\tau)
=q^4+3q^5+8q^6+5q^7+\cdots$$
and
$$h_2(\tau) =\eta^8(17\tau)\sum_{k=0}^7\mathcal E_{7\cdot 3^k}(\tau)\mathcal E_{3^k}^{-2}(\tau)\mathcal E^{-1}_{3^{k+1}}(\tau)
\mathcal E^{-1}_{5\cdot 3^{k}}(\tau)\mathcal E^{-1}_{8\cdot 3^{k}}(\tau)
=q^4+q^5+8q^6+\cdots,$$
where $\mathcal E_a(\tau)$ is given by \eqref{Ea}.
Then $$h_{17}(\tau)\eta^4(\tau)\eta^4(17\tau)=595h_1(\tau)-425h_2(\tau).$$ This gives the identity
\begin{align}\label{c-phi-17-right}\mathrm{C}\Phi_{17}(q) &= \frac{1}{(q^{17};q^{17})_\infty} + 17\sum_{j=1}^\infty p(17j-12)q^j+2\cdot 17^3 q^4\frac{(q^{17};q^{17})^5_\infty}{(q;q)_\infty^6} \nonumber\\
&\qquad\qquad +\frac{1}{q^3(q;q)_\infty^4(q^{17};q^{17})_\infty^5}\left(595h_1(\tau)-425h_2(\tau)\right).\end{align}
Note the simplicity of \eqref{c-phi-17-right} as compared to \eqref{CPhi-17}.
Identities similar to  \eqref{c-phi-17-right} exist for $k=19, 23$ and other primes. These identities involve the function
$\mathcal E_a(\tau)$.

\section{Generating function of $c\phi_{k}(n)$ for $k=9$ and $15$}

There are two cases to consider in this section, namely, $k=9 $ and 15.

\subsection{Case $k=9$}
\quad
\medskip

Let
\begin{align*}
E_{9,1}&=\frac{1}{240}+\sum_{k=1}^{\infty}\frac{k^3q^k}{1-q^k}, \\
E_{9,2}&=\frac{1}{240}+\sum_{k=1}^{\infty}\frac{k^3q^{3k}}{1-q^{3k}}, \\
E_{9,3}&=\frac{1}{240}+\sum_{k=1}^{\infty}\frac{k^3q^{9k}}{1-q^{9k}} \intertext{and}
E_{9,4}&=\sum_{n=1}^{\infty}\left(\frac{n}{3}\right)\sum_{d|n}d^3q^n.
\end{align*}
These are Eisenstein series of $M_{4}(\Gamma_{0}(9))$.

\begin{theorem}\label{cphi9-thm}
We have
\begin{align}\label{CPhi9-product}
\mathrm{C}\Phi_{9}(q)&=324q\frac{(q^3;q^3)_\infty^{8}}{(q;q)_\infty^{9}}+
19683q^{4}\frac{(q^9;q^9)_\infty^{12}}{(q;q)_\infty^{9}(q^3;q^3)_\infty^{4}}-240q\frac{(q^9;q^9)_\infty^{3}}{(q^3;q^3)_\infty^{4}}
\\ & \qquad\qquad -1458q^2\frac{(q^9;q^9)_\infty^{6}}{(q;q)_\infty^{3}(q^3;q^3)_\infty^{4}}
+\frac{(q;q)_\infty^{3}}{(q^3;q^3)_\infty^{4}}\notag\\ \label{Phi-8-E}
&= \frac{1}{(q;q)_\infty^9}\left(81E_{9,1}-84E_{9,2}+243E_{9,3}-3E_{9,4}-6q(q^3;q^3)_\infty^8\right).
\end{align}
\end{theorem}

\begin{proof}
By Theorem \ref{generating-thm}, we find that $\mathfrak{A}_9(q)\in M_{4}(\Gamma_{0}(9))$. Next, from \cite[Theorem 1.34]{Ono}, we find that $\dim M_{4}(\Gamma_{0}(9))=5$ and the basis is given by
\begin{align*}
B_{9,1}&=\eta^{8}(3\tau), \quad
B_{9,2}=\frac{\eta^{12}(9\tau)}{\eta^{4}(3\tau)}, \quad
B_{9,3}=\frac{\eta^{9}(\tau)\eta^{3}(9\tau)}{\eta^{4}(3\tau)},\\
B_{9,4}&=\frac{\eta^{6}(\tau)\eta^{6}(9\tau)}{\eta^{4}(3\tau)}, \quad
B_{9,5}=\frac{\eta^{12}(\tau)}{\eta^{4}(3\tau)}.
\end{align*}
By comparing Fourier coefficients of $\mathfrak{A}_{9}(q)$ and $B_{9,j}, 1\leq j\leq 5$, we deduce that
\begin{align}
\mathfrak{A}_{9}(q)=324B_{9,1}+19683B_{9,2}-240B_{9,3}-1458B_{9,4}+B_{9,5}.
\end{align}
This proves \eqref{CPhi9-product}.

We can replace the basis $\{B_{9,j}|1\leq j\leq 5\}$ by $\{B_{9,1}, E_{9,j}|1\leq j\leq 4\}$. Using these modular forms as a basis for $M_4(\Gamma_0(9))$, we deduce \eqref{Phi-8-E}.
\end{proof}

\begin{theorem}\label{cphi9-cong}
For any integer $n\ge 0$, we have
\begin{align}
c\phi_{9}(9n+3) &\equiv c\phi_{9}(9n+6) \equiv 0 \pmod{9}, \label{cphi9-mod9}\\
c\phi_{9}(3n+1) &\equiv 0 \pmod{81} \label{cphi9-mod81}\intertext{and}
c\phi_{9}(3n+2) &\equiv 0 \pmod{729}. \label{cphi9-mod729}
\end{align}
\end{theorem}
\begin{proof}
From \cite[Lemma 2.5]{WangIJNT}, we find that
\begin{align}\label{f1-3-dissection}
(q;q)_\infty^{3}=S(q^3)-3q(q^9;q^9)_\infty^{3},
\end{align}
where
\begin{align}\label{P-expan}
S(q)=(q;q)_\infty\left(\Theta_3(q)\Theta_3(q^3)+\Theta_2(q)\Theta_2(q^3) \right).
\end{align}
From \eqref{CPhi9-product}, we deduce that
\begin{align}
&\sum_{n=0}^\infty  c\phi_{9}(n)q^n \equiv 2^2\cdot 3^4q \frac{(q^3;q^3)_\infty^{8}}{(q;q)_\infty^{9}}-240q\frac{(q^9;q^9)_\infty^{3}}{(q^3;q^3)_\infty^{4}}+\frac{(q;q)_\infty^{3}}{(q^3;q^3)_\infty^{4}}
\pmod{729}  \nonumber \\
&\equiv 2^2\cdot 3^4  q(q^3;q^3)_\infty^{5}-240q\frac{(q^9;q^9)_\infty^{3}}{(q^3;q^3)_\infty^{4}}+\frac{S(q^3)}{(q^3;q^3)_\infty^{4}}-3q\frac{(q^9;q^9)_\infty^{3}}{(q^3;q^3)_\infty^{4}}  \pmod{729}. \label{mod-729-start}
\end{align}
Comparing the coefficients of $q^{3n+2}$ on both sides, we deduce that
$$c\phi_{9}(3n+2)\equiv 0 \pmod{729}.$$

Extracting the terms of the form $q^{3n+1}$ on both sides of \eqref{mod-729-start}, dividing by $q$ and replacing $q^3$ by $q$, we deduce that
\begin{align}
\sum_{n=0}^\infty c\phi_{9}(3n+1)q^n &\equiv 2^2\cdot 3^4 (q;q)_\infty^{5}-240\frac{(q^3;q^3)_\infty^{3}}{(q;q)_\infty^{4}}-3\frac{(q^3;q^3)_\infty^{3}}{(q;q)_\infty^{4}} \pmod{729} \nonumber \\
&\equiv 2^2 \cdot 3^4 (q;q)_\infty^{5}-243\frac{(q^3;q^3)_\infty^{3}}{(q;q)_\infty^{4}} \pmod{729}\nonumber \\
&\equiv 2^2 \cdot 3^4 (q;q)_\infty^{5}-3^{5}(q;q)_\infty^{5} \pmod{729}\nonumber \\
&\equiv 3^4 (q;q)_\infty^{5} \pmod{729},
\end{align}%
which implies \eqref{cphi9-mod81}.

Extracting the terms of the form $q^{3n}$ on both sides of \eqref{mod-729-start} and replacing $q^3$ by $q$, we find that
\begin{align}\label{mod-9-start}
\sum_{n=0}^\infty c\phi_{9}(3n)q^n &\equiv \frac{S(q)}{(q;q)_\infty^{4}} \pmod{729}\nonumber \\
&\equiv \frac{1}{(q;q)_\infty^{3}}\left( \Theta_3(q)\Theta_3(q^3)+\Theta_2(q)\Theta_2(q^3)\right) \pmod{729}.
\end{align}
From \cite[Lemma 2.6]{WangIJNT}, we deduce that
\begin{align}
\frac{1}{(q;q)_\infty^{3}}=\frac{(q^9;q^9)_\infty^{3}}{(q^3;q^3)_\infty^{12}}\left(S^2(q^3)+3qS(q^3)(q^9;q^9)_\infty^{3}+9q^2(q^9;q^9)_\infty^6\right). \label{id-1}
\end{align}
From  \cite[Corollary (i) and (ii), p.\ 49]{Berndt-notebook}, we find that
\begin{align}
&\Theta_3(q)=\Theta_3(q^9)+2qf(q^3,q^{15}), \label{id-2}\intertext{and}
&\Theta_2(q)=\Theta_2(q^9)+2q^{1/4}f(q^6,q^{12}), \label{id-4}
\end{align}
where $$f(a,b)=(-a;ab)_\infty(-b;ab)_\infty(ab;ab)_\infty.$$
Substituting \eqref{id-1}--\eqref{id-4} into \eqref{mod-9-start}, we deduce that
\begin{align}
&\sum_{n=0}^\infty c\phi_{9}(3n)q^n \equiv \frac{(q^9;q^9)_\infty^{3}}{(q^3;q^3)_\infty^{12}}
\Big(S^2(q^3)+3qS(q^3)(q^9;q^9)_\infty^{3}\Big) \nonumber\\
&\times\Bigg(\Theta_3(q^3)\Big(\Theta_3(q^9)+2qf(q^3,q^{15})\Big)+\Theta_2(q^3)\Big(\Theta_2(q^9)+2q^{1/4}f(q^6,q^{12})\Big)
\Bigg)
\pmod{9}. \label{cphi9-3dissection}
\end{align}

Extracting the terms of the form $q^{3n+1}$ on both sides of \eqref{cphi9-3dissection}, dividing by $q$ and replacing $q^3$ by $q$, applying \eqref{P-expan}, we deduce that
\begin{align}
&\sum_{n=0}^\infty c\phi_{9}(9n+3)q^{n} \notag\\
&\equiv S^{2}(q)\frac{(q^3;q^3)_\infty^{3}}{(q;q)_\infty^{12}}\left(3\frac{(q^3;q^3)_\infty^{3}}{(q;q)_\infty}+
2\left(\Theta_3(q)f(q,q^5)+q^{-1/4}\Theta_2(q)f(q^2,q^4)\right)\right)\!\!\!\!  \pmod{9} \notag\\
&\equiv S^2(q)\frac{(q^3;q^3)_\infty^3}{(q;q)_\infty^{12}} \notag\\
&\times\left(3\frac{(q^3;q^3)_\infty^{3}}{(q;q)_\infty}
+ 2\frac{(q^2;q^2)_\infty^{7}(q^3;q^3)_\infty (q^{12};q^{12})_\infty}{(q;q)_\infty^{3}(q^4;q^4)_\infty^{3}(q^6;q^6)_\infty}
+4\frac{(q^4;q^4)_\infty^{3}(q^6;q^6)_\infty^{2}}{(q^2;q^2)_\infty^{2}(q^{12};q^{12})_\infty}\right)\!\!\! \!\pmod{9}, \label{cphi9-mid}
\end{align}
where the last congruence follows by converting $$\Theta_3(q)f(q,q^5)+q^{-1/4}\Theta_2(q)f(q^2,q^4)$$ to infinite products.

From \cite[(3.75), (3.38)]{XiaYao2}, we find that
\begin{align}
\frac{(q^3;q^3)_\infty}{(q;q)_\infty^{3}}&=\frac{(q^4;q^4)_\infty^{6}(q^{6};q^{6})_\infty^{3}}{(q^{2};q^{2})_\infty^{9}(q^{12};q^{12})_\infty^{2}}+3q\frac{(q^4;q^4)_\infty ^{2}(q^{6};q^{6})_\infty(q^{12};q^{12})_\infty^{2}}{(q^{2};q^{2})_\infty^{7}} \label{1st-id} \intertext{and}
\frac{(q^3;q^3)_\infty^{3}}{(q;q)_\infty}&=\frac{(q^4;q^4)_\infty^{3}(q^{6};q^{6})_\infty^{2}}{(q^{2};q^{2})_\infty^{2}(q^{12};q^{12})_\infty}+q\frac{(q^{12};q^{12})_\infty^{3}}{(q^4;q^4)_\infty}. \label{2nd-id}
\end{align}
By \eqref{1st-id},  we find that
\begin{align}
&2\frac{(q^{2};q^{2})_\infty^{7}(q^3;q^3)_\infty (q^{12};q^{12})_\infty}{(q;q)_\infty^{3}(q^4;q^4)_\infty^{3}(q^{6};q^{6})_\infty}\\
&=2\frac{(q^{2};q^{2})_\infty^{7}(q^{12};q^{12})_\infty}{(q^4;q^4)_\infty^{3}(q^{6};q^{6})_\infty}\Big(\frac{(q^4;q^4)_\infty^{6}(q^{6};q^{6})_\infty^{3}}{(q^{2};q^{2})_\infty^{9}(q^{12};q^{12})_\infty^{2}}+3q\frac{(q^4;q^4)_\infty ^{2}(q^{6};q^{6})_\infty(q^{12};q^{12})_\infty^{2}}{(q^{2};q^{2})_\infty^{7}} \Big) \notag\\
&= 2\frac{(q^4;q^4)_\infty^{3}(q^{6};q^{6})_\infty^{2}}{(q^{2};q^{2})_\infty^{2}(q^{12};q^{12})_\infty}+6q\frac{(q^{12};q^{12})_\infty^{3}}{(q^4;q^4)_\infty}
\label{eq4-1}\end{align}
and this implies that
\begin{align}
&2\frac{(q^{2};q^{2})_\infty^{7}(q^3;q^3)_\infty (q^{12};q^{12})_\infty}{(q;q)_\infty^{3}(q^4;q^4)_\infty^{3}(q^{6};q^{6})_\infty}+
4\frac{(q^4;q^4)_\infty^{3}(q^{6};q^{6})_\infty^{2}}{(q^{2};q^{2})_\infty^{2}(q^{12};q^{12})_\infty}\notag \\
&= 6\frac{(q^4;q^4)_\infty^{3}(q^{6};q^{6})_\infty^{2}}{(q^{2};q^{2})_\infty^{2}(q^{12};q^{12})_\infty}+6q\frac{(q^{12};q^{12})_\infty^{3}}{(q^4;q^4)_\infty} \nonumber \\
&=6 \frac{(q^3;q^3)_\infty^{3}}{(q;q)_\infty}, \label{equality}
\end{align}
where we have used \eqref{2nd-id} in the last equality.
Substituting \eqref{equality} into \eqref{cphi9-mid}, we deduce that $$c\phi_{9}(9n+3)\equiv 0 \pmod{9}.$$

Extracting the terms of the form $q^{3n+2}$ on both sides of \eqref{cphi9-3dissection}, dividing by $q^2$ and replacing $q^3$ by $q$, we deduce that
\begin{align*}
&\sum_{n=0}^\infty c\phi_{9}(9n+6)q^n
\notag \\ & \equiv S(q)\frac{(q^3;q^3)_\infty^{6}}{(q;q)_\infty^{12}}\Big(6\frac{(q^{2};q^{2})_\infty^{7}(q^3;q^3)_\infty (q^{12};q^{12})_\infty}{(q;q)_\infty^{3}(q^4;q^4)_\infty ^{3}(q^{6};q^{6})_\infty}+3\frac{(q^4;q^4)_\infty^{3}(q^{6};q^{6})_\infty^{2}}{(q^{2};q^{2})_\infty^{2}(q^{12};q^{12})_\infty}  \Big)
\pmod{9}\\
&\equiv 0 \pmod{9},
\end{align*}
where we have used \eqref{eq4-1} to deduce the last congruence.
Hence $$c\phi_{9}(9n+6) \equiv 0 \pmod{9}.$$
\end{proof}

Congruences \eqref{cphi9-mod9} and \eqref{cphi9-mod81} can also be established using congruences discovered by Kolitsch.
In \cite{Kolitsch1987}, Kolitsch generalized Andrews' congruence \eqref{A-eq}
and proved that %. Let $\mu(x)$ be the classical M\"{o}bius function and $m$ be any integer. He proved that
\begin{align}\label{K-eq}
\sum_{d|(k,n)}\mu(d) c\phi_{\frac{k}{d}}(\frac{n}{d}) \equiv 0 \pmod{k^2},
\end{align}
where $\mu(n)$ is the M\"obius function (see for example \cite[Section 2.2]{Apostol}).
We now prove a generalization of \eqref{cphi9-mod9} and \eqref{cphi9-mod81}. For any non-negative integer $k$, we set $c\phi_{k}(x)=0$ whenever
$x\not\in\mathbf Z$. We can then rewrite \eqref{K-eq} as
\begin{align}\label{K-eq-2}
\sum_{d|k}\mu(d) c\phi_{\frac{k}{d}}\left(\frac{n}{d}\right) \equiv 0 \pmod{k^2}.
\end{align}

\begin{theorem}\label{general-cong-thm}
Let $p$ be a prime and $N$ be a positive integer which is not divisible by $p$. For any integers $\alpha \ge 1$ and $n\ge 0$, we have
\begin{equation}\label{general-cong}
c\phi_{p^{\alpha}N}(n) \equiv c\phi_{p^{\alpha-1}N}(n/p) \pmod{p^{2\alpha}},
\end{equation}
or equivalently,\begin{align}
c\phi_{p^{\alpha}N}(pn+r) &\equiv 0 \pmod{p^{2\alpha}}, \quad 1\le r \le p-1, \label{general-cong-1}\intertext{and}
c\phi_{p^{\alpha}N}(pn) &\equiv c\phi_{p^{\alpha-1}N}(n) \pmod{p^{2\alpha}}.\label{general-cong-2}
\end{align}
\end{theorem}
\begin{proof}
Let $\Omega(N)$ be the number of prime divisors of $N$ (counting multiplicities). We proceed by induction on $\Omega(N)$. If $\Omega(N)=0$, then $N=1$. Setting $k=p^{\alpha}$ in \eqref{K-eq-2}, we deduce that
\begin{align}\label{p-alpha-general}
c\phi_{p^{\alpha}}(n)\equiv c\phi_{p^{\alpha-1}}(n/p) \pmod{p^{2\alpha}}.
\end{align}
Thus, \eqref{general-cong} is true if $\Omega(N)=0$. Assume that \eqref{general-cong} is true if $\Omega(N)<h$, where $h$ is a positive integer. When $\Omega(N)=h$, we set $k=p^{\alpha}N$ in \eqref{K-eq-2}. Since $p$ does not divide $N$, any positive divisor of $p^{\alpha}N$ has the form $p^{j}d$ where $0\le j \le \alpha$ and $d|N$. In particular, if $j\ge 2$, then $\mu(p^{j}d')=0$. Hence by \eqref{K-eq-2}, we obtain
\begin{equation}\label{dNcphi}
\sum_{d|N} \left(\mu(d) c\phi_{\frac{p^{\alpha}N}{d}}\Big(\frac{n}{d}\Big)+\mu(pd) c\phi_{\frac{p^{\alpha-1}N}{d}}\Big(\frac{n}{pd}\Big)
\right)\equiv 0 \pmod{p^{2\alpha}}.
\end{equation}
According to $d=1$ or $d>1$, we separate the summands on the left hand side  of \eqref{dNcphi} and deduce that
\begin{align}\label{general-cong-start}
&c\phi_{p^{\alpha}N}(n)-c\phi_{p^{\alpha-1}N}\Big(\frac{n}{p}\Big)\\
&+\sum_{d|N,d>1}\mu(d)\Big(c\phi_{\frac{p^{\alpha}N}{d}}\Big(\frac{n}{d}\Big)-c\phi_{\frac{p^{\alpha-1}N}{d}}\Big(\frac{n}{pd}\Big) \Big) \equiv 0 \pmod{p^{2\alpha}}. \nonumber
\end{align}
Note that in the summand, since $d>1$, we have $\Omega(\frac{N}{d}) <h$ and hence by assumption,
\begin{equation}\label{gc-1} c\phi_{\frac{p^{\alpha}N}{d}}\Big(\frac{n}{d}\Big)-c\phi_{\frac{p^{\alpha-1}N}{d}}\Big(\frac{n}{pd}\Big) \equiv 0 \pmod{p^{2\alpha}}.\end{equation}
From \eqref{gc-1} and \eqref{general-cong-start}, we deduce that
\begin{align*}
c\phi_{p^{\alpha}N}(n)-c\phi_{p^{\alpha-1}N}\Big(\frac{n}{p}\Big) \equiv 0 \pmod{p^{2\alpha}}.
\end{align*}
Hence \eqref{general-cong} is true when $\Omega(N)=h$. This completes the proof of \eqref{general-cong}.

Replacing $n$ in \eqref{general-cong} by $pn+r$, where $0\le r \le p-1$, and observing that
\begin{align*}
c\phi_{p^{\alpha-1}N}\left(\frac{pn+r}{p} \right)=0, \quad 1\le r \le p-1,
\end{align*}
we deduce \eqref{general-cong-1} and \eqref{general-cong-2}.
\end{proof}

Let $(p,\alpha,N)=(3,2,1)$ in Theorem \ref{general-cong-thm}. By \eqref{general-cong-1}, we deduce that
\[c\phi_{9}(3n+1) \equiv c\phi_{9}(3n+2) \equiv 0 \pmod{81},\]
and this gives another proof of \eqref{cphi9-mod81}.  Similarly, by \eqref{general-cong-2}, we deduce that
\begin{align}\label{cphi9-add-2}
c\phi_{9}(3n)\equiv c\phi_{3}(n) \pmod{81}.
\end{align}
By \eqref{A-eq}, we find that
\[c\phi_{3}(3n+1) \equiv c\phi_{3}(3n+2) \equiv 0 \pmod{9}.\]
Substituting these congruences into \eqref{cphi9-add-2}, we complete the proof of \eqref{cphi9-mod9}.

\subsection{Case $k=15$}
\quad
\medskip

Let
\begin{align*}
f_{15}(\tau) &=\frac{\eta^2(\tau)\eta^2(15\tau)}{\eta(3\tau)\eta(5\tau)}, \\
h_{15}(\tau) &= \eta^4(\tau)\eta^4(5\tau)-9\eta^4(3\tau)\eta^4(15\tau),\\
g_{15,1}(\tau) &= -\frac{1}{8}\left(E_2(\tau)+3E_2(3\tau)-5E_2(5\tau)-15E_2(15\tau)\right),\\
g_{15,2}(\tau) &= -\frac{1}{12}\left(E_2(\tau)-3E_2(3\tau)+5E_2(5\tau)-15E_2(15\tau)\right),\\
g_{15,3}(\tau) &= \eta(\tau)\eta(3\tau)\eta(5\tau)\eta(15\tau),\intertext{and}
g_{15,4}(\tau) &= \frac{1}{8}\left(E_2(\tau)-3E_2(3\tau)-5E_2(5\tau)+15E_2(15\tau)\right),
\end{align*}
where
$$E_2(\tau) = 1-24\sum_{k=1}^\infty \frac{kq^k}{1-q^k}.$$

Using dimension formula \cite[Theorem 1.34]{Ono}, we find that
$$\text{dim}\,M_7\Big(\Gamma_0(15),\big(\frac{-15}{\cdot}\big)\Big)=14.$$
The  modular forms
\begin{alignat*}{3}
&B_{15,1} = f_{15}g_{15,1}^{3}, \qquad
&&B_{15,2}= f_{15} g_{15,1}^2g_{15,2},\qquad && \\
&B_{15,3} = f_{15}g_{15,1}g_{15,2}^2,
&&B_{15,4} = f_{15}g_{15,2}^3, && \\
&B_{15,5} = f_{15}g_{15,1}^2g_{15,3},
&&B_{15,6} = f_{15}g_{15,1}^2g_{15,4},  && \\
&B_{15,7} = f_{15}g_{15,1}g_{15,2}g_{15,3},\qquad
&& B_{15,8} = f_{15}g_{15,1}g_{15,2}g_{15,4}, && \\
&B_{15,9} = f_{15}g_{15,2}^2g_{15,3},\qquad
&& B_{15,10} = f_{15}g_{15,2}^2g_{15,4} &&
\\ & B_{15,11} = f_{15}g_{15,1}h_{15}
&& B_{15,12} =f_{15}g_{15,2}h_{15}, && \\
& B_{15,13} = \frac{\eta^{14}(3\tau)\eta^{14}(5\tau)}{\eta^{7}(\tau)\eta^{7}(15\tau)} \quad\text{and}\quad
&& B_{15,14} = \frac{\eta^{17}(\tau)\eta^{2}(5\tau)}{\eta^{4}(3\tau)\eta(15\tau)}
&&      \end{alignat*}
 form a basis for $M_7\Big(\Gamma_0(15),\big(\frac{-15}{\cdot}\big)\Big).$

Using the fact that $\mathfrak{A}_{15}(q)\in M_7\left(\Gamma_0(15),\left(\frac{-15}{\cdot}\right)\right)$, we deduce that
\begin{theorem}
For $|q|<1$, \begin{align*}
\mathrm{C}\Phi_{15}(q)
   =\frac{1}{(q;q)_{\infty}^{15}}\Big(
&\frac{18125225}{1156}B_{15,1} -\frac{845079}{34}B_{15,2} -\frac{87564447}{
  1156}B_{15,3} \\
  &+\frac{2491641}{34}B_{15,4} +\frac{147166525}{1156}B_{15,5}  +\frac{341957}{68}B_{15,6} \\
&-\frac{483081}{17}B_{15,7}-\frac{28623}{4}B_{15,8}
-\frac{9784683}{68}B_{15,9}\\
 &-\frac{1168839}{34}B_{15,10} +\frac{7263781}{68}B_{15,11}-\frac{97629}{4}B_{15,12}
\\ & +3375B_{15,13} -3374B_{15,14}\Big).
\end{align*}
\end{theorem}

\section{Generating function of $c\phi_{k}(n)$ for even integer $2<k<16$}

In this section, we derive alternative expressions for $\mathrm{C}\Phi_{k}(n)$ when $k>2$ is even.

\subsection{Case $k=4$}

\begin{theorem}
We have
\begin{align}
\mathrm{C}\Phi_{4}(q)=&\frac{1}{(q;q)_{\infty}^{4}}\Big(\Theta_3^{3}(q^2)+3\Theta_3(q^2)\Theta_2^2(q^2) \Big) \label{cphi4-product} \\
=&\frac{\Theta_3^{4}(q)}{(q;q)_{\infty}^{4}\Theta_3(q^2)}+\frac{\Theta_3^{2}(-q)\Theta_2^2(q^2) }{(q;q)_{\infty}^{4}\Theta_3(q^2)}. \label{cphi4-new}
\end{align}
\end{theorem}
\begin{proof}
Let $k=4$ in Theorem \ref{generating-thm}. We deduce that $\mathfrak{A}_{4}(q)\Theta_3(q) \in M_{2}\big(\Gamma_{0}(8), \big(\frac{2}{\cdot} \big) \big)$.
From \cite[Theorem 1.34]{Ono}, we deduce that
$$\dim M_{2}\big(\Gamma_{0}(8), \big(\frac{2}{\cdot} \big) \big)=3.$$
It can be verified that
\begin{align*}
\Theta_3(q)\Theta_3^{3}(q^2), \quad \Theta_3(q)^{3}\Theta_3(q^2), \quad \text{and}\quad \Theta_3(q^2)\Theta_2^2(q^2)\Theta_3(q)
\end{align*}
form a basis of $M_{2}\big(\Gamma_{0}(8), \big(\frac{2}{\cdot} \big) \big)$. Comparing the Fourier coefficients of $\mathfrak{A}_{4}(q)\Theta_3(q)$ and the given basis of  $M_{2}\big(\Gamma_{0}(8), \big(\frac{2}{\cdot} \big) \big)$, we deduce that
\begin{equation*}
\mathfrak{A}_{4}(q)\Theta_3(q)=\Big(\Theta_3^{3}(q^2)+3\Theta_3(q^2)\Theta_2^2(q^2) \Big)\Theta_3(q),
\end{equation*}
which  proves (\ref{cphi4-product}).

Theorem \ref{generating-thm} also implies that $\Theta_3(q^2) \mathfrak{A}_{4}(q)\in M_{2}\big(\Gamma_{0}(16)\big)$. From \cite[Theorem 1.34]{Ono}, we find that $\dim M_{2}\big(\Gamma_{0}(16)\big)=5$. Identity \eqref{cphi4-new} then follows from the fact that
\begin{align*}
\Theta_3^{4}(q), \quad \Theta_3^{4}(q^2), \quad \Theta_3^{4}(q^4),\quad \Theta_3^{2}(-q)\Theta_3^{2}(-q^2), \quad
\text{and}\quad \Theta_3^{2}(-q)\Theta_2^2(q^2)
\end{align*}
form a basis of $M_{2}\big(\Gamma_{0}(16)\big)$.
\end{proof}

\begin{rem}
The representation \eqref{cphi4-new} was first deduced by W. Zhang and C. Wang \cite{Zhang-Wang} from \eqref{cphi4-product}, where they used it to give an elementary proof of the congruence
\[
c\phi_{4}(7n+5) \equiv 0 \pmod{7}.
\]
\end{rem}

\subsection{Case $k=6$}
\quad
\medskip

\begin{theorem}\label{cphi6-thm}
We have
\begin{align}\label{cphi6-product}
\mathrm{C}\Phi_{6}(q)=&\frac{4}{9}\frac{(q;q)_\infty^{5}(q^4;q^4)_\infty ^{2}}{(q^{2};q^{2})_\infty^{5}(q^3;q^3)_\infty^{3}}-\frac{1}{3}\frac{(q^{2};q^{2})_\infty^{4}(q^4;q^4)_\infty^{2}}{(q;q)_\infty^{4}(q^{6};q^{6})_\infty^{3}}+\frac{8}{9}\frac{(q^4;q^4)_\infty ^{11}}{(q;q)_\infty^{4}(q^{2};q^{2})_\infty^{5}(q^{12};q^{12})_\infty^{3}} \nonumber \\
&+36q\frac{(q^4;q^4)_\infty^{2}(q^3;q^3)_\infty^{9}}{(q;q)_\infty^{7}(q^{2};q^{2})_\infty^{5}}+27q^2\frac{(q^4;q^4)_\infty ^{2}(q^{6};q^{6})_\infty^{9}}{(q;q)_\infty^{4}(q^{2};q^{2})_\infty^{8}}\nonumber \\
&+72q^4\frac{(q^{12};q^{12})_\infty^{9}}{(q;q)_\infty^{4}(q^{2};q^{2})_\infty^{5}(q^4;q^4)_\infty}.
\end{align}
\end{theorem}
\begin{proof}
Let $k=6$ in Theorem \ref{generating-thm}. We deduce that $\Theta_3(q)\mathfrak{A}_6(q) \in M_{3}\big(\Gamma_{0}(12), \big(\frac{-12}{\cdot} \big) \big)$. From \cite[Theorem 1.34]{Ono}, we deduce that
$$\dim M_{3}\Big(\Gamma_{0}(12), \big(\frac{-12}{\cdot} \big) \Big)=7.$$
Let
\begin{alignat*}{4}
&B_{6,1}=\frac{\eta^{9}(\tau)}{\eta^{3}(3\tau)},\quad
&&B_{6,2}=\frac{\eta^{9}(2\tau)}{\eta^{3}(6\tau)},\quad
&&B_{6,3}=\frac{\eta^{9}(4\tau)}{\eta^{3}(12\tau)},
&&B_{6,4}=\frac{\eta^{9}(3\tau)}{\eta^{3}(\tau)}, \\
&B_{6,5}=\frac{\eta^{9}(6\tau)}{\eta^{3}(2\tau)},
&&B_{6,6}=\frac{\eta^{9}(12\tau)}{\eta^{3}(4\tau)}\quad\text{and}\quad
&&B_{6,7}=\eta^{3}(2\tau)\eta^{3}(6\tau). &&
\end{alignat*}
The set $\{B_{6,j}|1\leq j\leq 7\}$ forms a basis of $M_{3}\big(\Gamma_{0}(12), \big(\frac{-12}{\cdot} \big) \big)$ and by comparing the Fourier coefficients of $\Theta_3(q)\mathfrak{A}_6(q)$ and modular forms in $\{B_{6,j}|1\leq j\leq 7\}$, we deduce that
\begin{align}
\Theta_3(q)\mathfrak{A}_{6}(z)=\frac{4}{9}B_{6,1}-\frac{1}{3}B_{6,2}+\frac{8}{9}B_{6,3}+36B_{6,4}+27B_{6,5}+72B_{6,6}.
\end{align}
This proves \eqref{cphi6-product}.
\end{proof}

Congruences for $c\phi_{6}(n)$ have drawn much attention in recent years. For example, Baruah and Sarmah \cite{Baruah-Rama} established 3-dissections of $\mathrm{C}\Phi_{6}(q)$ and proved that
\begin{align}
c\phi_{6}(3n+1) &\equiv 0 \pmod{9} \label{cphi-6-mod9-1}\intertext{and}
c\phi_{6}(3n+2) &\equiv 0 \pmod{9}. \label{cphi-6-mod9-2}
\end{align}
We remark here that the congruences above follow directly from \eqref{general-cong-1} with $(p,\alpha,N)=(3,1,2)$. Moreover, setting $(p,\alpha,N)=(2,1,3)$ in \eqref{general-cong-1}, we deduce that %the apparently new congruence
\begin{align}\label{cphi6-2n+1}
c\phi_{6}(2n+1) \equiv 0 \pmod{4}.
\end{align}
Congruence \eqref{cphi6-2n+1} appeared in \cite{Baruah-Rama} as Corollary 3.1.

For more congruences satisfied by $c\phi_{6}(n)$, see a recent paper of C. Gu, L. Wang and E.X.W. Xia \cite{GuWangXia} and their list of references.

\subsection{Case $k=8$}
\quad
\medskip
\begin{theorem}
We have
\begin{align}\label{cphi8-product}
&\mathrm{C}\Phi_{8}(q)=\frac{1}{(q;q)_{\infty}^{8}}
\Big(\Theta_3^{7}(q^4)+28\Theta_3^6(q^4)\Theta_2(q^4)+105 \Theta_3^5(q^4) \Theta_2^2(q^4)\\
&+112\Theta_3^4(q^4)\Theta_2^3(q^4)+147\Theta_3^3(q^4)\Theta_2^4(q^4)+84\Theta_3^2(q^4)\Theta_2^5(q^4) +35\Theta_3(q^4)\Theta_2^6(q^4)  \Big). \notag
\end{align}
\end{theorem}
\begin{proof}
Let $k=8$ in Theorem \ref{generating-thm}. We deduce that $\Theta_3(q)\mathfrak{A}_{8}(q) \in M_{4}(\Gamma_{0}(16))$. From \cite[Theorem 1.34]{Ono}, we find that
$$\dim M_{4}(\Gamma_{0}(16))=9$$
and one can verify that
\begin{alignat*}{3}
&B_{8,1}=\Theta_3(q)\Theta_3^7(q^4), \quad &&B_{8,2}=\Theta_3(q)\Theta_3^6(q^4)\Theta_2(q^4),\quad &&B_{8,3}=\Theta_3(q)\Theta_3^5(q^4)\Theta_2^2(q^4), \\ &B_{8,4}= \Theta_3(q)\Theta_3^4(q^4)\Theta_2^3(q^4),\quad &&B_{8,5}= \Theta_3(q)\Theta_3^3(q^4)\Theta_2^4(q^4), &&B_{8,6}=  \Theta_3(q)\Theta_3^2(q^4)\Theta_2^5(q^4),  \\
&B_{8,7}= \Theta_3(q)\Theta_2^7(q^4), &&B_{8,8}=\Theta_2^8(q^4),\quad\quad \text{and}\quad  &&B_{8,9}= \Theta_3(q)\Theta_3(q^4)\Theta_2^6(q^4).
\end{alignat*}
form a basis for $M_{4}(\Gamma_{0}(16))$. By comparing the Fourier coefficients of the basis and those of $\Theta_3(q)\mathfrak{A}_{8}(q)$, we find that
\begin{align}\label{cphi8-theta}
\Theta_3(q)\mathfrak{A}_{8}(q)=&B_{8,1}+28B_{8,2}+105B_{8,3}+112B_{8,4}+147B_{8,5}+84B_{8,6}+35B_{8,9}.
\end{align}
This completes the proof of \eqref{cphi8-product}.
\end{proof}

By \eqref{p-alpha-general}, we find that
\begin{align}\label{cphi-8-relate}
c\phi_{8}(n)\equiv c\phi_{4}(n/2) \pmod{64}.
\end{align}

In \cite{BaruahDisc}, Baruah and Sarmah proved that
\begin{align}\label{BS-1}
c\phi_{4}(2n+1)&\equiv 0 \pmod{4^2}, \\
c\phi_{4}(4n+2) &\equiv 0 \pmod{4},\label{BS-2} \intertext{and}
c\phi_{4}(4n+3) &\equiv 0 \pmod{4^4}.\label{BS-3}
\end{align}
Combining \eqref{BS-1}--\eqref{BS-3} with \eqref{cphi-8-relate}, we obtain the following congruences for $c\phi_{8}(n)$:
\begin{theorem}\label{cphi-8-mod2thm}
For any integer $n\ge 0$,
\begin{align}
c\phi_{8}(2n+1) &\equiv 0 \pmod{64}, \\
c\phi_{8}(4n+2) &\equiv 0 \pmod{16}, \\
c\phi_{8}(8n+4) &\equiv 0 \pmod{4}\intertext{and}
c\phi_{8}(8n+6) &\equiv 0 \pmod{64}.
\end{align}
\end{theorem}

\subsection{Case $k=10$}
\quad\medskip

By Theorem \ref{generating-thm}, we have $\Theta_3(q)\mathfrak{A}_{10}(q) \in M_{5}\big(\Gamma_{0}(20),(\frac{-20}{\cdot})\big)$. From \cite[Theorem 1.34]{Ono}, we deduce that
$$\dim M_{5}\big(\Gamma_{0}(20),(\frac{-20}{\cdot})\big)=14.$$

Let
\begin{alignat*}{2}
&B_{10,1}=\Theta_{3}^9(q)\Theta_{3}(q^5), &&\,\,
B_{10,2}=\Theta_{3}(q)\Theta_{2}^3(q^{1/2})\Theta_{2}^3(q)\Theta_{2}^3(q^{5/2}), \\
&B_{10,3}=\Theta_{3}(q)\Theta_{3}^{3}(q^5)\Theta_{2}^2(q^{1/2})\Theta_{2}^2(q)\Theta_{2}^2(q^{5/2}), &&\,\,B_{10,4}=\Theta_{3}(q)\Theta_{3}(q^5)\Theta_{2}^2(q^{1/2})\Theta_{2}^6(q^{5/2}),\\
&B_{10,5}=\Theta_{3}(q)\Theta_{3}(q^5)\Theta_{2}^8(q), &&\,\,B_{10,6}=\Theta_{3}^7(q)\Theta_{3}^3(q^5),\\
&B_{10,7}=\Theta_{3}^6(q^5)\Theta_{2}^3(q^{1/2})\Theta_{2}(q^{5/2}), &&\,\,B_{10,8}=\Theta_{3}(q)\Theta_{3}^5(q^5)\Theta_{2}^4(q),\\
&B_{10,9}=\Theta_{3}^3(q^5)\Theta_{2}(q^{1/2})\Theta_{2}^5(q)\Theta_{2}(q^{5/2}), &&\,\,B_{10,10}=\Theta_{3}^6(q)\Theta_{2}^3(q)\Theta_{2}(q^{5}),\\
&B_{10,11}=\Theta_{3}(q)\Theta_{3}(q^5)\Theta_{2}^8(q^{1/2}),  &&\,\,B_{10,12}=\Theta_{3}(q)\Theta_{3}(q^5)\Theta_{2}^6(q)\Theta_{2}^2(q^{5}),\\
&B_{10,13}=\Theta_{3}^5(q)\Theta_{3}^5(q^5), \quad\quad\qquad \text{and}\quad &&\,\,B_{10,14}=\Theta_{3}^3(q)\Theta_{3}^3(q^5)\Theta_{2}^3(q^{1/2})\Theta_{2}(q^{5/2}).
\end{alignat*}
The set $\{B_{10,j}|1\le j \le 14\}$ forms a basis of $M_{5}\big(\Gamma_{0}(20),(\frac{-20}{\cdot}) \big)$ and we deduce the following:
\begin{theorem}\label{cphi-10}
We have
\begin{align}
&\mathrm{C}\Phi_{10}(q)=\frac{1}{\Theta_3(q)(q;q)_{\infty}^{10}}\Big(\frac{13}{8}B_{10,1}+\frac{435}{32}B_{10,2}+ \frac{9275}{128}B_{10,3}+ \frac{175}{32}B_{10,4}-\frac{31}{8}B_{10,5} \nonumber\\
&-\frac{15}{4}B_{10,6}  +\frac{225}{4}B_{10,7}-\frac{775}{32}B_{10,8}+\frac{221}{32} B_{10,10}
-\frac{857}{512}B_{10,11}+\frac{155}{8}B_{10,12}
+\frac{25}{8}B_{10,13}\Big). \label{cphi-10-formula}
\end{align}
\end{theorem}

Let
\begin{align*}f_{10}&=\frac{\eta(\tau)\eta(2\tau)\eta(10\tau)\eta(20\tau)}{\eta(4\tau)\eta(5\tau)}\\
f_{10,1}&= \Theta_3(q)\Theta_3(q^5)\frac{\eta^{10}(\tau)}{\eta^{2}(5\tau)}\intertext{and}
f_{10,2}&=\Theta_3^6(q^5)\frac{\eta^5(20\tau)}{\eta(4\tau)}.\end{align*}
Let
\begin{align*}%
g_{10,1} &= \frac{1}{6}\left(E_2(\tau)-4E_2(2\tau)+4E_2(4\tau)+5E_2(5\tau)-20E_2(10\tau)+20E_2(20\tau)\right),\\
g_{10,2} &= \Theta_3^2(q)\Theta_3^2(q^5),\\
g_{10,3}&=\frac{1}{4}\left(-E_2(2\tau)+5E_2(10\tau)\right),\\
g_{10,4} &= -\frac{1}{24}\left(E_2(\tau)+E_2(2\tau)+4E_2(4\tau)-5E_2(5\tau)-5E_2(10\tau)-20E_2(20\tau)\right),\\
g_{10,5}&=\eta^2(2\tau)\eta^2(10\tau)\intertext{and}
g_{10,6}&=\frac{5}{4}\Theta_3^4(q^5)-\frac{1}{4}\Theta_3^4(q).
\end{align*}
Let
\begin{alignat*}{2}
&B^*_{10,1}=f_{10}g_{10,1}^2, \quad && B^*_{10,2}=f_{10}g_{10,1}g_{10,2},\\
& B^*_{10,3}=f_{10}g_{10,2}^2, && B^*_{10,4}=f_{10}g_{10,1}g_{10,3},\\
& B^*_{10,5}=f_{10}g_{10,1}g_{10,4},\quad && B^*_{10,6}=f_{10}g_{10,1}g_{10,5}, \\
& B^*_{10,7}=f_{10}g_{10,1}g_{10,6}, \quad && B^*_{10,8}=f_{10}g_{10,2}g_{10,3},\\
& B^*_{10,9}=f_{10}g_{10,2}g_{10,4},\quad  && B^*_{10,10}=f_{10}g_{10,2}g_{10,5}, \\
& B^*_{10,11}=f_{10}g_{10,2}g_{10,6},\quad  && B^*_{10,12}=f_{10}g_{10,3}^{2},\\
& B^*_{10,13}=f_{10,1},\quad\quad\text{and}\quad && B^*_{10,14}=f_{10,2}.
\end{alignat*}
We can replace
the basis $\{B_{10,j}|1\le j \le 14\}$  by
the basis $\{B^*_{10,j}|1\le j \le 14\}$  and deduce that
\begin{align}\label{CPhi10-B*}
\mathrm{C}\Phi_{10}(q)=&\frac{1}{\Theta_3(q)(q;q)_{\infty}^{10}}\Big(\frac{5075}{2}B^*_{10,1}+\frac{4525}{4}B^*_{10,2}+ \frac{29375}{4}B^*_{10,3}+ \frac{4525}{2}B^*_{10,4}-4525B^*_{10,5} \nonumber\\
&-6525B^*_{10,6}  +\frac{6275}{4}B^*_{10,7}-4950B^*_{10,8}+2300B^*_{10,9}-22375 B^*_{10,10} \notag \\
&+\frac{10325}{4}B^*_{10,11}-10150B^*_{10,12}+B^*_{10,13}+200000B^*_{10,14}\Big).
\end{align}
Identity \eqref{CPhi10-B*} leads immediately to
\begin{align}\label{CPhi-10-mod5}
{\rm C}\Phi_{10}(q) \equiv \frac{\Theta_3(q^5)}{(q^5;q^5)_\infty^2} \pmod{5^2}.
\end{align}

\begin{rem} Congruence \eqref{CPhi-10-mod5} is the motivation behind the discovery of Theorem \ref{general-cong-thm}.
Theorem \ref{general-cong-thm}, when interpreted in terms of generating functions, yield the congruence
\begin{equation}\label{CPhi-pell} {\rm C}\Phi_{p\ell}(q) \equiv {\rm C}\Phi_{\ell}(q^p)  \pmod{p^2}\end{equation} for any
distinct primes $p$ and $\ell$. Congruence \eqref{CPhi-10-mod5} is a special case of \eqref{CPhi-pell} once we identify
the right hand side of \eqref{CPhi-10-mod5} with $\mathrm{C}\Phi_{2}(q^5)$ (see \eqref{c-phi-2-mod}).
\end{rem}

\begin{theorem}\label{cphi-10-mod5thm}
For any integer $n\ge 0$, we have
\begin{align}
c\phi_{10}(2n+1) &\equiv 0 \pmod{4}, \label{cphi-10-mod4} \\
c\phi_{10}(5n+r) &\equiv 0\pmod{5^2}, \quad 1 \le r \le 5 \label{cphi-10-mod5}
\intertext{and}
\label{cphi-10-mod5new}
c\phi_{10}(25n+15) &\equiv 0 \pmod{5}.
\end{align}
\end{theorem}
\begin{proof}
Congruences \eqref{cphi-10-mod4} and \eqref{cphi-10-mod5} follow from Theorem \ref{general-cong-thm} by setting $(p,\alpha,N)=(2,1,5)$ and $(5,1,2)$, respectively. Congruence \eqref{cphi-10-mod5} also follows from \eqref{CPhi-10-mod5}. Furthermore, from \eqref{CPhi-10-mod5}, we deduce that
\begin{align}\label{cphi-10-mod5new-start}
\sum_{n=0}^\infty & c\phi_{10}(5n)q^n \equiv \frac{\Theta_3(q)}{(q;q)_{\infty}^{2}}
\equiv \frac{\Theta_3(q)(q;q)_{\infty}^{3}}{(q;q)_{\infty}^{5}} \pmod{5}
\\
&\equiv \frac{1}{(q^5;q^5)_{\infty}}\left( \sum_{i=-\infty}^\infty \sum_{j=0}^{\infty}(-1)^j(2j+1)q^{i^2+j(j+1)/2}\right) \pmod{5}.\notag
\end{align}
Note that
\[n=i^2+\frac{j(j+1)}{2} \,\,\text{if and only if}\,\, 8n+1=8i^2+(2j+1)^2.\]
Since $\left(\dfrac{-8}{5} \right)=-1$, we find that
$8n+1 \equiv 0\pmod{5}$, or equivalently, that
$$n\equiv 3\pmod{5} \,\,\text{if and only if}\,\, i\equiv 2j+1 \equiv 0\pmod{5}.$$
 Hence, by \eqref{cphi-10-mod5new-start}, we deduce that $$c\phi_{10}(5(5n+3)) \equiv 0\pmod{5}.$$
\end{proof}

\begin{rem} One can prove \eqref{cphi-10-mod5new} by first observing that \eqref{CPhi-10-mod5} implies
\[c\phi_{10}(5n) \equiv c\phi_{2}(n) \pmod{5^2}.\]
Using \eqref{Andrews-cphi-2-cong}, we deduce \eqref{cphi-10-mod5new}.
\end{rem}

\subsection{Case $k=12$}
\quad\medskip

By Theorem \ref{generating-thm}, we have $\Theta_3(q)\mathfrak{A}_{12}(q) \in M_{6}\big( \Gamma_{0}(24), (\frac{24}{\cdot})\big)$. By \cite[Theorem 1.34]{Ono}, we deduce that
$$\dim M_{6}\big( \Gamma_{0}(24), (\frac{24}{\cdot})\big)=22.$$

Let
{\small
\begin{alignat*}{2}
&B_{12,1}=\Theta_{3}^3(q)\Theta_{3}^9(q^6),  && B_{12,2}=\Theta_{3}^{11}(q)\Theta_{3}(q^6), \\
&B_{12,3}=\Theta_{3}(q^2)\Theta_{3}(q^3)\Theta_{2}^{10}(q^2),
&&B_{12,4}=\Theta_{3}^2(q)\Theta_{3}(q^2)\Theta_{3}(q^3)\Theta_{2}^4(q)\Theta_{2}^4(q^3), \\ &B_{12,5}=\Theta_{3}^2(q)\Theta_{3}(q^2)\Theta_{3}^5(q^3)\Theta_{2}^2(q)\Theta_{2}^2(q^3), \quad && B_{12,6}=\Theta_{3}^{2}(q)\Theta_{3}(q^2)\Theta_{3}^9(q^3),  \\
&B_{12,7}=\Theta_{3}(q)\Theta_{3}^{3}(q^6)\Theta_{2}^4(q^2)\Theta_{2}^4(q^{6}), \quad && B_{12,8}=\Theta_{3}(q)\Theta_{3}^5(q^2)\Theta_{2}^3(q^2)\Theta_{2}^{3}(q^{6}), \\
&B_{12,9}=\Theta_{3}(q)\Theta_{3}(q^2)\Theta_{2}^5(q^2)\Theta_{2}^5(q^{6}),  \quad && B_{12,10}=\Theta_{3}(q)\Theta_{3}^5(q^2)\Theta_{2}^3(q)\Theta_{2}^3(q^3),  \\
&B_{12,11}=\Theta_{3}(q)\Theta_{3}(q^2)\Theta_{2}^5(q)\Theta_{2}^5(q^3),  \quad &&
B_{12,12}=\Theta_{3}(q)\Theta_{3}^7(q^6)\Theta_{2}^2(q)\Theta_{2}^2(q^3), \\
&B_{12,13}=\Theta_{3}(q)\Theta_{3}^3(q^6)\Theta_{2}^4(q)\Theta_{2}^4(q^3),
&&B_{12,14}=\Theta_{3}(q)\Theta_{3}^5(q^2)\Theta_{2}(q)\Theta_{2}^2(q^2)\Theta_{2}(q^3)\Theta_{2}^2(q^{6}),  \\
&B_{12,15}=\Theta_{3}(q)\Theta_{3}(q^2)\Theta_{2}(q)\Theta_{2}^8(q^2)\Theta_{2}(q^3),
&&B_{12,16}=\Theta_{3}(q)\Theta_{3}^3(q^2)\Theta_{2}(q)\Theta_{2}^6(q^2)\Theta_{2}(q^3),  \\
&B_{12,17}=\Theta_{3}(q)\Theta_{3}^5(q^2)\Theta_{2}(q)\Theta_{2}^4(q^2)\Theta_{2}(q^3),  \quad &&  B_{12,18}=\Theta_{3}(q)\Theta_{3}^9(q^2)\Theta_{2}(q)\Theta_{2}(q^3),  \\
&B_{12,19}=\Theta_{3}(q)\Theta_{3}^2(q^2)\Theta_{3}(q^6)\Theta_{2}^8(q^2), \quad  && B_{12,20}=\Theta_{3}(q)\Theta_{3}^6(q^2)\Theta_{3}(q^6)\Theta_{2}^4(q^2),\\
&B_{12,21}=\Theta_{3}(q)\Theta_{3}^8(q^2)\Theta_{3}(q^6)\Theta_{2}^2(q^2), \quad \text{and}\quad &&
B_{12,22}=\Theta_{3}(q)\Theta_{3}^{10}(q^2)\Theta_{3}(q^6).
\end{alignat*}
}
The set $\{B_{12,j}|1\le j\le 22\}$ forms a basis of $M_{6}(\Gamma_{0}(24), (\frac{24}{\cdot}))$.
Using the above basis, we deduce the following identity:
\begin{theorem}\label{c-phi-12}
We have
\begin{align}\label{cphi-12-formula}
\mathrm{C}\Phi_{12}(q)=&\frac{1}{\Theta_3(q)(q;q)_{\infty}^{12}}\Big(-\frac{36207}{160}B_{12,1}+\frac{923091}{4000}B_{12,4}+\frac{35829}{100}B_{12,5}+\frac{891}{4}B_{12,6}\nonumber\\
&-\frac{1485}{8}B_{12,7}-\frac{143247}{1000}B_{12,8}-\frac{891}{4}B_{12,9}-\frac{8109}{160}B_{12,10} \nonumber \\
&-\frac{582717}{16000}B_{12,11}+\frac{227691}{200}B_{12,12}+\frac{714249}{8000}B_{12,13}+\frac{8109}{80}B_{12,14} \nonumber\\
&+\frac{33}{8}B_{12,15}+\frac{294109}{500}B_{12,16}-\frac{16503}{400}B_{12,17}-\frac{99}{8}B_{12,18}+\frac{10559}{200}B_{12,19}\nonumber \\
&-\frac{128807}{100}B_{12,20} +\frac{25647}{160}B_{12,21}+\frac{727}{160}B_{12,22}\Big). \end{align}
\end{theorem}

Next, we give some congruences satisfied by $c\phi_{12}(n)$.
\begin{theorem}\label{cphi-12-cong-thm}
We have
\begin{align}
c\phi_{12}(2n+1) &\equiv 0 \pmod{16}, \label{cphi-12-mod16} \\
c\phi_{12}(3n+1) &\equiv 0 \pmod{9} \label{cphi-12-mod9-1} \intertext{and}
c\phi_{12}(3n+2) &\equiv 0 \pmod{9}. \label{cphi-12-mod9-2}
\end{align}
\end{theorem}
\begin{proof}
This follows directly from Theorem \ref{general-cong-thm} by setting $(p,\alpha,N)=(2,2,3)$ and $(3,1,4)$.
\end{proof}

\subsection{Case $k=14$}
\quad\medskip

By Theorem \ref{generating-thm}, we know $\Theta_3(q)\mathfrak{A}_{14}(q)\in M_{7}\big(\Gamma_{0}(28), (\frac{-28}{\cdot}) \big)$. By \cite[Theorem 1.34]{Ono}, we deduce that
$$\dim M_{7}\big(\Gamma_{0}(28), (\frac{-28}{\cdot}) \big)=27.$$

Let
{\small
\begin{alignat*}{2}
&B_{14,1}=\Theta_{3}^{13}(q)\Theta_{3}(q^7), \quad
 && B_{14,2}=\Theta_{3}^{7}(q)\Theta_{3}^3(q^7)\Theta_{2}^2(q^{1/2})\Theta_{2}^2(q^{7/2}), \\
&B_{14,3}=\Theta_{3}^5(q)\Theta_{3}(q^7)\Theta_{2}^4(q^{1/2})\Theta_{2}^4(q^{7/2}), \quad && B_{14,4}=\Theta_{3}(q)\Theta_{3}^5(q^7)\Theta_{2}^4(q^{1/2})\Theta_{2}^4(q^{7/2}),\\
&B_{14,5}=\Theta_{3}^5(q)\Theta_{3}^3(q^7)\Theta_{2}^2(q)\Theta_{2}^4(q^{7/2}),  \quad && B_{14,6}=\Theta_{3}^3(q)\Theta_{3}^7(q^7)\Theta_{2}^2(q^{1/2})\Theta_{2}^2(q^{7/2}),\\
&B_{14,7}=\Theta_{3}^{11}(q)\Theta_{3}^3(q^7), \quad && B_{14,8}=\Theta_{3}^{12}(q)\Theta_{2}(q^{1/2})\Theta_{2}(q^{7/2}),\\
&B_{14,9}=\Theta_{3}^8(q)\Theta_{2}^3(q^{1/2})\Theta_{2}^3(q^{7/2}),  \quad && B_{14,10}=\Theta_{3}^4(q)\Theta_{2}^5(q^{1/2})\Theta_{2}^5(q^{7/2}),\\
&B_{14,11}=\Theta_{3}(q)\Theta_{3}(q^5)\Theta_{2}^8(q^{1/2}), \quad && B_{14,12}=\Theta_{3}^{12}(q)\Theta_{2}(q)\Theta_{2}(q^{7}),\\
&B_{14,13}=\Theta_{3}^{8}(q)\Theta_{2}^{3}(q)\Theta_{2}^{3}(q^{7}),\quad && B_{14,14}=\Theta_{3}^{4}(q)\Theta_{2}^{5}(q)\Theta_{2}^5(q^{7}),\\
&B_{14,15}=\Theta_{2}^7(q)\Theta_{2}^{7}(q^{7}),  \quad && B_{14,16}=\Theta_{3}^{8}(q)\Theta_{3}(q^7)\Theta_{2}^{2}(q^{1/2})\Theta_{2}^{3}(q),\\
&B_{14,17}=\Theta_{3}^6(q)\Theta_{3}^{3}(q^7)\Theta_{2}^2(q^{1/2})\Theta_{2}^3(q), \quad && B_{14,18}=\Theta_{3}^8(q)\Theta_{2}(q)\Theta_{2}^{4}(q^{2})\Theta_{2}(q^7),\\
&B_{14,19}=\Theta_{3}^{2}(q)\Theta_{3}^{7}(q^7)\Theta_{2}^{2}(q^{1/2})\Theta_{2}^{3}(q), \quad &&
B_{14,20}=\Theta_{3}^{9}(q^7)\Theta_{2}^{2}(q^{1/2})\Theta_{2}^{3}(q),\\
&B_{14,21}=\Theta_{3}^{10}(q)\Theta_{3}(q^7)\Theta_{2}(q)\Theta_{2}^{2}(q^{7/2}), \quad &&
B_{14,22}=\Theta_{3}^{4}(q)\Theta_{3}(q^7)\Theta_{2}^{3}(q)\Theta_{2}^{6}(q^{7/2}),\\
&B_{14,23}=\Theta_{3}^{2}(q)\Theta_{3}^{3}(q^7)\Theta_{2}^{3}(q)\Theta_{2}^{6}(q^{7/2}), \quad &&
B_{14,24}=\Theta_{3}^{5}(q^7)\Theta_{2}^3(q)\Theta_{2}^6(q^{7/2}),\\
&B_{14,25}=\Theta_{3}^{11}(q)\Theta_{2}(q^{1/2})\Theta_{2}(q)\Theta_{2}^{3}(q^{7/2}),
&&B_{14,26}=\Theta_{3}^{2}(q)\Theta_{3}^4(q^7)\Theta_{2}(q^{1/2})\Theta_{2}^{2}(q)\Theta_{2}^5(q^{7/2}),\intertext{and}
&B_{14,27}=\Theta_{3}^4(q)\Theta_{3}^2(q^7)\Theta_{2}(q^{1/2})\Theta_{2}^2(q)\Theta_{2}^5(q^{7/2}). &&
\end{alignat*}
}
The set $\{B_{14,j}|1\le j\le 27\}$ forms a basis of $M_{7}\big(\Gamma_{0}(28), (\frac{-28}{\cdot})\big)$.
This basis allows us to derive the following identity:
\begin{theorem}\label{c-phi-14}
We have
\begin{align}\label{cphi-14-formula}
\mathrm{C}\Phi_{14}(q)=&\frac{1}{\Theta_3(q)(q;q)_{\infty}^{14}}\Big(
-\frac{3}{4}B_{14,1} -\frac{332339}{1024}B_{14,2}+\frac{255927}{4096}B_{14,3} -\frac{197519}{4096}B_{14,4}\nonumber \\
&+\frac{17325}{64}B_{14,5}+\frac{1407329}{2048}B_{14,6}+\frac{7}{4}B_{14,7} +\frac{3}{4}B_{14,8}
-\frac{13765}{256}B_{14,9}\nonumber\\
&-\frac{52045}{1024}B_{14,10}+\frac{3861}{512}B_{14,12}+\frac{429}{16}B_{14,13}+\frac{6623}{64}B_{14,16}
-\frac{79799}{512}B_{14,17}\nonumber \\
&+\frac{29407}{512}B_{14,19} -\frac{3989}{64}B_{14,21}+\frac{19803}{128}B_{14,22} -\frac{16807}{256}B_{14,23} \nonumber\\
& +\frac{50421}{256}B_{14,26}-\frac{6895}{256}B_{14,27} \Big).
\end{align}
\end{theorem}
By setting $(p,\alpha,N)=(7,1,2)$ in \eqref{general-cong}, we get
\begin{align}\label{cphi-14-relate}
c\phi_{14}(n) \equiv c\phi_{2}(n/7) \pmod{49},
\end{align}
By \eqref{general-cong-1}, we deduce that
\begin{align}
c\phi_{14}(7n+r) &\equiv 0 \pmod{49}, \quad 1\le r \le 6.
\end{align}
Moreover, setting $(p,\alpha,N)=(2,1,7)$ in \eqref{general-cong-1}, we deduce that
\begin{align}
c\phi_{14}(2n+1) \equiv 0 \pmod{4}.
\end{align}

\subsection{Case $k=16$}
\quad\medskip

By Theorem \ref{generating-thm}, we know $\Theta_3(q)\mathfrak{A}_{16}(q)\in M_{8}\big(\Gamma_{0}(32), (\frac{-2}{\cdot}) \big)$. By \cite[Theorem 1.34]{Ono}, we deduce that
$$\dim M_{8}\big(\Gamma_{0}(32), (\frac{-2}{\cdot}) \big)=32.$$

Let
\begin{align*}
&B_{16,j}=\Theta_{3}(q)\Theta_{3}^{j-1}(q^2)\Theta_{2}^{16-j}(q^{8}), \quad 1\le j \le 15, \\
&B_{16,j}=\Theta_{3}^{3}(q)\Theta_{3}^{j-16}(q^2)\Theta_{2}^{29-j}(q^{8}), \quad 16 \le j \le 29, \\
&B_{16,30}=\Theta_{3}^5(q)\Theta_{3}^2(q^2)\Theta_{3}^9(q^8),\\
&B_{16,31}=\Theta_{3}^9(q)\Theta_{3}(q^4)\Theta_{3}^6(q^8), \intertext{and}
&B_{16,32}=\Theta_{3}^3(q)\Theta_{3}(q^4)\Theta_{3}(q^8)\Theta_{2}^{11}(q^{8}).
\end{align*}
The set $\{B_{16,j}|1\le j \le 32\}$ forms a basis of $M_{8}\big(\Gamma_{0}(32), (\frac{-2}{\cdot}) \big)$. Hence, we deduce the following identity:
\begin{theorem}\label{add-thm-cphi-16}
We have
\begin{align}\label{correct-cphi-16}
\mathrm{C}\Phi_{16}(q)=&\frac{1}{\Theta_3(q)(q;q)_{\infty}^{16}}\Big( -16384B_{16,1}+ 122880B_{16,2} -431024B_{16,3}\nonumber\\
&+10384B_{16,4}+ 3956568B_{16,5} -12663584B_{16,6}\nonumber \\
&+ 21477101B_{16,7} -23125005B_{16,8}+15986724B_{16,9} \nonumber \\
&-6153988B_{16,10}+ 108966B_{16,11}+1259002B_{16,12}-678464B_{16,13}\nonumber \\
&+ 162042B_{16,14} -15218B_{16,15}+61440B_{16,18} -337920B_{16,19} \nonumber \\
&+ 844918B_{16,20}-870438B_{16,21}-327528B_{16,22}+1914696B_{16,23}  \nonumber \\ &-2366700B_{16,24}+1511404B_{16,25}-484664B_{16,26}+34128B_{16,27}\nonumber \\
&+20722B_{16,28} -58B_{16,29}+59B_{16,30}\Big).
\end{align}
\end{theorem}

By Theorem \ref{general-cong-thm}, we obtain
\begin{align}
c\phi_{16}(2n+1) \equiv 0 \pmod{256}
\end{align}
and
\begin{align}
c\phi_{16}(2n) \equiv c\phi_{8}(n) \pmod{256}.
\end{align}
\section{M\"obius inversion and Kolitsch's congruence \eqref{K-eq}}

In this section, we will use a different notation for $k$-colored  generalized Frobenius symbol $\lambda$.
The color of a part will be placed on the left
hand side of the part. In other words, our symbol $\lambda$ is now written as
\begin{align}\label{pi-symbol}
\lambda =\begin{pmatrix} c_1(z_1) & c_2(z_2) &\cdots &c_d(z_d) \\
 c'_1(z'_1) & c'_2(z'_2) &\cdots &c'_d(z'_d)\end{pmatrix},
\end{align}
 where
 $c_j$ and $c'_j$ denote colors from the set $\{1,2,\cdots, k\}$ and $z_j, z'_j$ denote the parts. For example, the 2-colored generalized Frobenius symbol
 \begin{align*}
\begin{pmatrix} 2_2 & 2_1 \\  1_2 & 0_1 \end{pmatrix}
 \end{align*}
 is now written as
 \begin{align*}
 \begin{pmatrix}
 2(2) & 1(2)  \\
 2(1) & 1(0)
 \end{pmatrix}.
 \end{align*}

Let $\sigma_k$ be the $k$-cycle $(1\quad 2\quad \cdots \quad k).$
Let the symbol $$\left(\begin{matrix} \cdots \cdots   \\  \cdots \cdots   \end{matrix}\right)^{\text{sort}}$$ denote
sorting the resulting rows to be strictly decreasing %after the application of a permutation
according to \eqref{new-ordering}.
We say that $\lambda$ has order $\ell$ with respect to $\sigma_k$ if $\ell$ is the smallest positive integer for which the equality of the following symbols holds:
\begin{align*}
\left(\begin{matrix} c_1(z_1)\ & c_2(z_2) & \cdots & c_d(z_d) \\
c'_1(z'_1) & c'_2(z'_2) &\cdots  &c'_d(z'_d)
\end{matrix}\right) =\left(\begin{matrix} \sigma_k^\ell(c_{1})(z_1)  &\sigma_k^\ell(c_{2})(z_2) &\cdots  &\sigma_k^\ell(c_{d})(z_d) \\
\sigma_k^\ell(c'_{1})(z'_1)  &\sigma_k^\ell(c'_{2})(z'_2) &\cdots  &\sigma_k^\ell(c'_{d})(z'_d)
\end{matrix}\right)^{\text{sort}}.
\end{align*}
For example, with respect to the 4-cycle $\sigma_4=(1\quad 2 \quad 3 \quad 4)$, the 4-colored generalized Frobenius symbol
\begin{align*}
\begin{pmatrix}
3(1) & 1(1) \\
4(2) & 2(2)
\end{pmatrix}
\end{align*}
has order 2 while
\begin{align*}
\begin{pmatrix}
2(1) & 1(1) \\
4(2) & 2(2)
\end{pmatrix}
\end{align*}
has order 4.

Let $\Psi_{k,\ell}(n)$ be the number of $k$-colored generalized Frobenius symbols of $n$ that have order $\ell.$
When $\ell=k$, we follow Kolitsch and denote $\Psi_{k,k}(n)$ by $\overline{c\phi}_k(n)$. For example, we have $\overline{c\phi}_{2}(2)=8$ since there are eight $2$-colored generalized Frobenius symbols of 2 that have order 2:
\begin{align*}
&\begin{pmatrix} 1(1) \\ 1(0) \end{pmatrix}, \quad \begin{pmatrix} 2(1) \\ 1(0) \end{pmatrix}, \quad \begin{pmatrix} 1(1) \\ 2(0) \end{pmatrix}, \quad \begin{pmatrix} 2(1) \\ 2(0) \end{pmatrix}, \\
&\begin{pmatrix} 1(0) \\ 1(1) \end{pmatrix}, \quad \begin{pmatrix} 1(0) \\ 2(1) \end{pmatrix}, \quad \begin{pmatrix} 2(0) \\ 1(1) \end{pmatrix}, \quad \begin{pmatrix} 2(0) \\ 2(1) \end{pmatrix}.
\end{align*}

The function $\overline{c\phi}_k(n)$ is implicitly mentioned by Kolitsch in \cite{Kolitsch1987} and the following
identity was later given by him in \cite[p. 220]{Kolitsch1989}:

\begin{theorem}\label{Kolitsch-k-color} Let $k$ and $n$ be positive integers. Then
\begin{equation}\label{Kolitsch-Dirichlet}
\overline{c\phi}_k(n)=\sum_{\ell|(k,n)}\mu(\ell)  c\phi_{\frac{k}{\ell}}\left(\frac{n}{\ell}\right).
\end{equation}
\end{theorem}
With \eqref{Kolitsch-Dirichlet}, \eqref{K-eq} can be written as
\begin{equation}\label{Kmain1}
\overline{c\phi}_k(n) \equiv 0 \pmod{k^2}.
\end{equation}
Congruence \eqref{Kmain1} provides an elegant analogue of Andrews' orginal congruence \eqref{A-eq}, which states
that
$$c\phi_p(n)\equiv 0 \pmod{p^2}$$ for primes $p$ not dividing $n$.
Using the definition of $\overline{c\phi}_k(n)$, we can rewrite
  \eqref{Kolitsch5}, \eqref{Kolitsch7} and \eqref{Analogue11}
\cite[Theorem 3]{Kolitsch1991} as
 \begin{equation*}
 \overline{c\phi}_5(n)=5p(5n-1), \overline{c\phi}_7(n)=7p(7n-2) \,\,\text{and}\,\,
 \overline{c\phi}_{11}(n)=11p(11n-5)\end{equation*}
 where $n$ is any positive integer.

In this section,  we prove the following:
\begin{theorem} \label{Kolitsch-k-color-inverse}
Let $k$ and $n$ be positive integers. Then
\begin{equation}\label{Kolitsch-Dirichlet-inverse} c\phi_k(n)=\sum_{\ell|k}
\overline{c\phi}_{\ell}\left(\frac{n}{(k/\ell)}\right)=\sum_{\ell|k}
\overline{c\phi}_{k/\ell}\left(\frac{n}{\ell}\right).\end{equation}
\end{theorem}
We then establish \eqref{Kolitsch-Dirichlet} using Theorem \ref{Kolitsch-k-color-inverse}.
We will also take this opportunity to present Kolitsch's proof of \eqref{Kmain1} (see Theorem \ref{Kolitsch-best}). Our presentation of Kolitsch's proof contains more details than that given in \cite{KolitschDM1990}. We feel that it is important for us (and perhaps the reader)
to fully understand Koltisch's proof as it is an important congruence
and that it is essential in our proof of Theorem \ref{general-cong-thm}.

We now begin our proof of Theorem \ref{Kolitsch-k-color-inverse}.
\begin{proof}[Proof of Theorem \ref{Kolitsch-k-color-inverse}]
Every $k$-colored generalized Frobenius symbol has an order $\ell$ with respect to $\sigma_k$.
We first show that the order of a $k$-colored generalized Frobenius symbol $\lambda$ must divide $k$.
Suppose not. Let $m=ds$ be the order of $\lambda$ with $d=(m,k)$ and $s>1$.
Observe that $\sigma_k^d$ splits into a product of $d$ disjoint cycles $C_j$,
$1\leq j\leq d$, of length $k/d$. Since $(s,k/d)=1$,
$\left(\sigma_k^d\right)^s$ is again a product of $d$ disjoint cycles $C'_j$, $1\leq j\leq d$, and the integers in $C_j^s$
are the same as those in $C_j$.
Hence, if $\sigma_k^m$ leaves $\lambda$ invariant,
it would have been left invariant under $\sigma_k^d$ but this contradicts the minimality of $m$.
Therefore, the order of $\lambda$ must be a divisor of $k$ and we deduce that
$$c\phi_k(n)=\sum_{\ell|k} \Psi_{k,\ell}(n).$$
To prove \eqref{Kolitsch-Dirichlet-inverse}, it suffices to show that
\begin{equation}\label{K1}\Psi_{k,\ell}(n) = \overline{c\phi}_{\ell}\left(\frac{n}{(k/\ell)}\right).\end{equation}
For $\ell|k$, we know that
$\sigma_k^\ell$ splits into $\ell$ disjoint cycles $C_j$, $j=1,2,\cdots,\ell$ of length $k/\ell$.
Now, if $\lambda$ is a
$k$-colored generalized Frobenius symbol of order $\ell$, then it means that
if an entry $c_\nu(z)$, with $c_\nu$ appearing in $C_j$, appears in $\lambda$
%and
then $c_\mu(z)$ must appear in $\lambda$ for every
color $c_\mu$ that appears in the cycle $C_j.$
We now replace all the colors in this cycle where $c_\nu$ belongs by the color represented by the smallest
integer, which can be chosen to be less than $\ell$. In this way, we will obtain a $\ell$-colored
generalized Frobenius symbol where each entry $c_j(z)$ appears $k/\ell$ times.
In other words, from
\begin{equation*} n = d +\sum_{i=1}^{d}c_{i}(z_i)
+ \sum_{i=1}^{d} c'_{i}(z'_i),\end{equation*}
we obtain a $\ell$-colored generalized Frobenius symbol giving rise the partition
\begin{equation*} n
=d+  \frac{k}{\ell}\left(\sum_{i=1}^{d/(k/\ell)}c_{j_i}(z_i) +  \sum_{i=1}^{d/(k/\ell)} c'_{j_i}(z'_i)\right),\end{equation*}
which implies that
\begin{equation*}\frac{n}{(k/\ell)}
=\dfrac{d}{(k/\ell)} +  \sum_{i=1}^{d/(k/\ell)}c_{j_i}(z_i) +  \sum_{i=1}^{d/(k/\ell)} c'_{j_i}(z'_i).\end{equation*}
We have thus constructed from $\lambda$ the $\ell$-colored generalized Frobenius symbol of $n/(k/\ell)$, which we denote as $\lambda^*$. We claim that $\lambda^*$ has order
$\ell$ with respect to $$\gamma=(1\quad 2\quad \cdots \quad \ell).$$
If $\lambda^*$ is of order $m$ less than $\ell$, then this means that
$$\gamma^m =\prod_{j=1}^{m} C'_j,$$
where each $C'_j$ is a $\ell/m$ cycle,
leaves  $\lambda^*$ invariant.
Since $m<\ell$, at least two of the integers $u$ and $v$ between $1$ and $\ell$ are in some cycle
$C'_j$. When we reverse the above process of obtaining $\ell$-colored generalized Frobenius symbol of $n/(k/\ell)$ from a $k$-colored generalized Frobenius symbol of $n$ of order $\ell$, we  would obtain a symbol $\lambda$ which is
 fixed by  a cycle that includes both $u$ and $v$. But $u$ and $v$ are in disjoint cycles in the decomposition of
$\sigma_k^\ell$ and this contradicts the fact that $\lambda$ has order $\ell$.
Hence,
$\lambda^*$ cannot have  order strictly less than $\ell$ and its order must be $\ell$.

Conversely, given a $\ell$-colored generalized Frobenius symbol of $n/(k/\ell)$ of order $\ell$ with respect to $\gamma$, we reverse
the process to obtain a $k$-colored generalized Frobenius symbol of $n$ of order $\ell$.
Hence, we have \eqref{K1} and the proof of Theorem \ref{Kolitsch-k-color-inverse} is complete.
\end{proof}

\medskip
Theorem \ref{Kolitsch-k-color} now follows from Theorem \ref{Kolitsch-k-color-inverse} by using the following lemma with
$F(n,k)=c\phi_k(n)$ and $G(n,k)=\overline{c\phi}_k(n)$:

\begin{lemma}\label{Dirichlet}
Let $F(n,k)$ and $G(n,k)$ be two-variable arithmetical functions. Then
\begin{align}\label{F-G}
F(n,k) = \sum_{\ell|(n,k)} G(n/\ell, k/\ell),
\end{align}
if and only if
\begin{align}\label{G-F}
G(n,k) = \sum_{\ell|(n,k)} \mu(\ell) F(n/\ell, k/\ell).
\end{align}
\end{lemma}

\begin{proof}
To prove \eqref{G-F}, we set $n=dn'$ and $k=dk'$ where $d=(n,k)$.
From \eqref{F-G}, we
have
$$F(n'd, k'd) = \sum_{\ell|d} G(n'd/\ell, k'd/\ell).$$
Using the M\"obius inversion formula, we deduce that
$$G(n'd,k'd) = \sum_{\ell|d} \mu(\ell) F(n'd/\ell, k'd/\ell),$$ or
$$G(n,k)=
\sum_{\ell|d} \mu(\ell) F(n/\ell,k/\ell).$$

The converse follows in a similar way from the M\"obius inversion formula.
\end{proof}

\begin{rem}
We observe that using the above inversion, we can find an expression of M\"obius function in terms of Ramanujan's sum
$c_q(n).$ We will write Ramanujan's sum as $c(q,n)$. It is known that \cite[Section 8.3]{Apostol}
$$c(q,n) = \sum_{\ell|(q,n)} \mu(q/\ell) \ell.$$
Now, we observe that
$$\frac{c(q,n)}{(q,n)} = \sum_{\ell|(q,n)} \mu(q/\ell)\frac{\ell}{(q,n)}.$$
Using the inversion formula with
$$F(q,n) = \frac{c(q,n)}{(q,n)} \quad\text{and}\quad G(q,n) =\frac{\mu(q)}{(q,n)},$$ we deduce that
$$\frac{\mu(q)}{(q,n)} =\sum_{\ell|(q,n)} \frac{c(q/\ell,n/\ell)}{(q,n)}\ell\mu(\ell),$$
or
$${\mu(q)}=\sum_{\ell|(q,n)} c(q/\ell, n/\ell) \ell \mu(\ell).$$
\end{rem}

\begin{theorem}\label{Kolitsch-best} Let $k$ and $n$ be positive integers. Then
\begin{equation*} \overline{c\phi}_k(n) \equiv 0 \pmod{k^2}.\end{equation*}
\end{theorem}
We will next prove Theorem \ref{Kolitsch-best}.
\begin{proof}[Proof of Theorem \ref{Kolitsch-best}]
Given a $k$-colored generalized Frobenius symbol $\lambda$ represented by \eqref{pi-symbol}, we say that the color difference of $\lambda$ is $m$ when $m$ is the
sum of the numerical values of the colors on the first row minus the sum of the numerical values of  the colors on the second row of $\lambda$.  In other words,
\[m=c_1+c_2+\cdots +c_d-(c_1'+c_2'+\cdots +c_d').\]
Let $\overline{c\phi}_k(m,n)$ denote the number of $k$-colored generalized Frobenius symbol $\lambda$ of $n$ with color difference $m$ and order $k$. Let $c\phi_k(m,n)$ denote the number of $k$-colored generalized Frobenius symbol $\lambda$ of $n$ with color difference $m$. These functions satisfy the following analogue of \eqref{Kolitsch-Dirichlet-inverse}:
\begin{equation}\label{Kolitsch-Dirichlet-inverse2}
c\phi_k(m,n) = \sum_{\ell|k}\overline{c\phi}_{\ell}\left(\frac{m}{k/\ell},\frac{n}{k/\ell}\right).
\end{equation}

The proof of \eqref{Kolitsch-Dirichlet-inverse2} is the same as \eqref{Kolitsch-Dirichlet-inverse} by checking that there is an one to one correspondence between a $k$-colored generalized Frobenius symbol of $n$ with color difference $m$ and order $\ell$ and a $\ell$-colored
generalized Frobenius symbol of $n/(k/\ell)$ with color difference $m/(k/\ell)$ and order $\ell$.
The only additional step we need to observe is that under our previous construction, when we replace the $k$-colored generalized Frobenius symbol $\lambda$ of $n$ with a $k$-colored generalized Frobenius symbol $\lambda^\dagger$ with only colors $j$ with $1\leq j\leq \ell$ (by identifying colors belong to the cycle containing $j$), the
color difference of $\lambda^\dagger$
becomes  $m/(k/\ell)$. This is because if a color $j$ appears in $\lambda$, then the rest of the colors belonging to the
cycle containing $j$ are of the form $j+w\ell$, $1\leq w< k/\ell$.

Using inversion formula
similar to Lemma \ref{Dirichlet} with two-variable arithmetical functions replaced by three-variable arithmetical functions, we deduce
 from \eqref{Kolitsch-Dirichlet-inverse2} that
\begin{equation}\label{Kolitsch-Dirichlet2}
\overline{c\phi}_k(m,n) = \sum_{\ell|k}\mu(\ell)c\phi_{k/\ell}\left(\frac{m}{\ell},\frac{n}{\ell}\right).
\end{equation}

Now, the function
$$\sum_{m=-\infty}^\infty \sum_{n=0}^\infty c\phi_k(m,n)t^mq^n$$
is the constant term, i.e., coefficient of $z^0$ of the function
$$\prod_{j=1}^k (zt^jq;q)_\infty (z^{-1}t^{-j};q)_\infty,$$
which we shall write as
\begin{equation}\label{Gen-CT}\sum_{m=-\infty}^\infty \sum_{n=0}^\infty c\phi_k(m,n)t^mq^n=\text{CT}\left(\prod_{j=1}^k (zt^jq;q)_\infty (z^{-1}t^{-j};q)_\infty\right).\end{equation}
See \cite[pp. 4--6, Theorems 5.1 and 5.2]{Andrews} for examples of expressing generating functions of various partitions functions as constant term of infinite products involving $z$.

From \eqref{Kolitsch-Dirichlet2} and \eqref{Gen-CT}, we deduce that
\begin{align}\sum_{m=-\infty}^\infty\sum_{n=0}^\infty \overline{c\phi}_k(m,n)t^mq^n
&=\sum_{\ell|k}\mu(\ell) \sum_{m=-\infty}^\infty \sum_{n=0}^\infty c\phi_{k/\ell}\left(\frac{m}{\ell},\frac{n}{\ell}\right)t^mq^n\notag\\
&=\sum_{\ell|k}\mu(\ell) \sum_{m=-\infty}^\infty \sum_{n=0}^\infty c\phi_{k/\ell}\left(m,n\right)t^{\ell m}q^{\ell n}\notag\\
&=\sum_{\ell|k}\mu(\ell)\text{CT}\left(\prod_{j=1}^{k/\ell} (zt^{\ell j}q^\ell;q^\ell)_\infty (z^{-1}t^{-\ell j};q^\ell)_\infty\right)\notag\\
&=\sum_{\ell|k}\mu(\ell)\text{CT}\left(\prod_{j=1}^{k/\ell} (z^\ell t^{\ell j}q^\ell;q^\ell)_\infty (z^{-\ell}t^{-\ell j};q^\ell)_\infty\right),
\label{CT1}
\end{align}
where the last equality follows from the fact that
\eqref{Gen-CT} holds with  $z$ replaced by $z^a$ for any positive integer $a$.

Next, we rewrite the left hand side of \eqref{Gen-CT} as
\begin{align}
\sum_{m=-\infty}^\infty\sum_{n=0}^\infty \overline{c\phi}_k(m,n)t^mq^n &=\sum_{j=0}^{k-1}\sum_{s=-\infty}^\infty{\sum_{n=0}^{\infty} }\overline{c\phi}_k(sk+j,n)t^{sk+j}
q^n \notag \\
&=\sum_{n=0}^{\infty} \left(\sum_{j=0}^{k-1} \sum_{s=-\infty}^\infty \overline{c\phi}_k(sk+j,n) t^{sk+j}\right) q^n.\label{alt1}
\end{align}
Let $$c_k(j,n) = \sum_{\substack{m=-\infty \\ m\equiv j\!\!\!\!\pmod{k}}}^\infty\overline{c\phi}_k(m,n).$$
Let $t=1$ in \eqref{alt1}.   Note that
\[\overline{c\phi}_{k}(n)=\sum_{m=-\infty}^{\infty}\overline{c\phi}_{k}(m,n).\]
We find that
\begin{equation}\label{Firsteq}\sum_{j=0}^{k-1}c_k(j,n) = \overline{c\phi}_k(n).\end{equation}

Next, if $t=\zeta\neq 1$ is a primitive $r$-th root of unity with $r |k$, then from \eqref{alt1}, we deduce that
\begin{equation}\label{add-eq-ckj}
\sum_{m=-\infty}^\infty\sum_{n=0}^\infty \overline{c\phi}_k(m,n)\zeta^mq^n
%=\sum_{j=0}^{k-1} \zeta^j\sum_{n=0}^\infty c_k(j,n)q^n
=\sum_{n=0}^\infty \sum_{j=0}^{k-1}c_k(j,n )\zeta^j q^n.
\end{equation}

To complete the proof of \eqref{Kmain1}, we need the following lemma:

\begin{lemma} \label{imL} Let $\zeta_k$ be a primitive $k$-th root of unity.
Then $\zeta_k^s$ is a root of
$$P_{n}(t):=\sum_{j=0}^{k-1}c_k(j,n)t^j$$
for all $1\leq s\leq k-1$.
\end{lemma}

Assuming that Lemma \ref{imL} is true.  It would imply that $P_{n}(t)$ is divisible by $Q(t)=1+t+\cdots +t^{k-1}$ and
since the degrees of $P_{n}(t)$ and $Q(t)$ are the same, we must conclude that
$c_k(j,n)=c_k(0,n)$ are all equal for $1\leq j\leq k-1$. From \eqref{Firsteq}, we conclude that
$$\overline{c\phi}_k(n) = kc_k(0,n).$$ Let
$S_0$ be the set of $k$-colored generalized Frobenius symbols of $n$ of order $k$ with color difference divisible by $k$. Note that $|S_0|=c_k(0,n)$.
If $\pi\in S_0$ then $\pi$ under the action of $\sigma_k^j, 1\leq j\leq k-1$ is also in $S_0$ since {  the residue of} the color difference is invariant {  modulo $k$} under the action of $\sigma$ and the order of $\pi$ is $k$. This implies that $S_0$
  can be grouped into disjoint sets containing $k$ elements in each set, which implies that $k$ divides $c_k(0,n)$. Therefore,
$$\overline{c\phi}_k(n) \equiv 0 \pmod{k^2}$$
and this completes the proof of \eqref{Kmain1}.
\end{proof}

It remains to prove Lemma \ref{imL}.
\begin{proof}[Proof of Lemma \ref{imL}]
Given any integer $j$ between 1 and $k-1$, there exists an integer $r|k$ such that $\zeta_{k}^{j}$ is a primitive $r$-th root of unity.
Therefore, to prove Lemma \ref{imL},
it suffices to prove that $P_{n}(\zeta)=0$ for any primitive $r$-th root of unity with $r|k$.
From \eqref{CT1}  and \eqref{add-eq-ckj},
we deduce that
\begin{equation}\label{sum-mu}
\sum_{n=0}^{\infty}P_{n}(\zeta)q^n=\sum_{\ell|k}\mu(\ell)\text{CT}\left(\prod_{j=1}^{k/\ell} (z^\ell \zeta^{\ell j}q^\ell;q^\ell)_\infty (z^{-\ell}\zeta^{-\ell j};q^\ell)_\infty\right).
\end{equation}
The presence of the factor $\mu(\ell)$ in \eqref{sum-mu} shows that we only need to consider divisors of the squarefree part of $k$.
Fix a prime $p$ which divides $r$ and separate the sum in \eqref{sum-mu} into a sum over divisors of the form $d$ where $(p,d)=1$ and
a sum over divisors of the form $pd$.
We only need to show that the term corresponding to $d$ cancels with the term corresponding to $pd.$

Observe that since $d$ is squarefree and $(d,p)=1$, we can write $d=ww'$ where $w|r$ and $(w',r)=1$.
Note that the term corresponding to $d=ww'$ is
\begin{align*}&\text{CT}\left(\mu(ww')\prod_{j=1}^{k/(ww')} (z^{ww'} \zeta^{ww' j}q^{ww'};q^{ww'})_\infty (z^{-ww'}\zeta^{-ww' j};q^{ww'})_\infty\right)\\
&=\mu(ww')\text{CT}\left((z^{rw'} q^{rw'};q^{rw'})_\infty^{k/(rw')} (z^{-rw'};q^{rw'})_\infty^{k/(rw')}\right)\end{align*}
since $\zeta^w$ is a $r/w$-th primitive root of unity and
$$\prod_{j=0}^{\nu} (1-z\zeta_\nu^j) = (1-z^\nu).$$
Similarly, the term corresponding to $pd=pww'$ is
\begin{align*}&\text{CT}\left(\mu(pww')\prod_{j=1}^{k/(pww')} (z^{pww'} \zeta^{pww' j}q^{pww'};q^{pww'})_\infty (z^{-pww'}\zeta^{-pww' j};q^{pww'})_\infty\right)\\
&=\mu(pww')\text{CT}\left((z^{rw'} q^{rw'};q^{rw'})_\infty^{k/(rw')} (z^{-rw'};q^{rw'})_\infty^{k/(rw')}\right).
\end{align*}
Clearly these two terms cancel as $\mu(pww')=-\mu(ww')$.
This completes the proof of the lemma.
\end{proof}

\noindent {\it Acknowledgement.} We would like to express our sincere thanks to
 Professor George E. Andrews for his encouragement and strong support of
 our project.
This work is completed during the first author's stay at the Faculty
of Mathematics, University of Vienna. The first author would like to
thank his host Professor C.~Krattenthaler for his hospitality and for
providing an excellent research environment during his stay in Vienna.

The second author was supported by the National Natural Science Foundation of China (11801424), the Fundamental Research Funds for the Central Universities (Grant No.\ 1301--413000053) and a start-up research grant (No.\ 1301--413100048) of the Wuhan University. The third author is partially supported by Grant 102-2115-M-009гн001-MY4 of the Ministry of Science and Technology,  Taiwan (R.O.C.)

\section*{Corrigendum after publication}
This paper has been published as \cite{CWY-TAMS}. On May 24 and 25, 2021, Dr.\ Dazhao Tang informed the second author that there might be some typos in equations \eqref{cphi-12-formula} and \eqref{correct-cphi-16} in the published version \cite{CWY-TAMS}. The second author then checked again by \emph{Mathematica} and found the following typos:
\begin{enumerate}
\item In Theorem \ref{c-phi-12}, i.e., equation \eqref{cphi-12-formula}, the coefficient for $B_{12,11}$ should be $-\frac{582717}{16000}$, and the coefficient for $B_{12,16}$ should be $\frac{294109}{500}$. In the published version \cite{CWY-TAMS}, these were wrongly typed as $-\frac{582717}{4000}$ and $\frac{1179561}{4000}$.
\bigskip

\item In Theorem \ref{add-thm-cphi-16}, i.e., equation \eqref{correct-cphi-16}, the coefficient for $B_{16,23}$ should be $1914696$.  In the published version \cite{CWY-TAMS}, this was wrongly typed as $122540544$. By the way, the power of $(q;q)_\infty$ should be $16$ instead of $15$.
\end{enumerate}

We sincerely thank Dr.\ Dazhao Tang for pointing out these typos.

In addition, on page 16, we find that in the definition of $g_{17,2}(\tau)$, the factor $\frac{1}{4}$ was missing in the published version \cite{CWY-TAMS}.

\end{document}